%% file: Lyap_Damping_Arxiv.tex
\documentclass[%
]{mpi2015-cscpreprint}
\usepackage[american]{babel}

\usepackage{amssymb}
\usepackage{amsthm}
\usepackage{multirow}

\usepackage{mathtools}
\usepackage{nccmath}
\usepackage{graphicx}
\usepackage[caption=false]{subfig} %
\usepackage{float}

\makeatletter
\newcommand*\rel@kern[1]{\kern#1\dimexpr\macc@kerna}
\newcommand*\widebar[1]{%
  \begingroup
  \def\mathaccent##1##2{%
    \rel@kern{0.8}%
    \overline{\rel@kern{-0.8}\macc@nucleus\rel@kern{0.2}}%
    \rel@kern{-0.2}%
  }%
  \macc@depth\@ne%
  \let\math@bgroup\@empty\let\math@egroup\macc@set@skewchar%
  \mathsurround\z@ \frozen@everymath{\mathgroup\macc@group\relax}%
  \macc@set@skewchar\relax
  \let\mathaccentV\macc@nested@a%
  \macc@nested@a\relax111{#1}%
  \endgroup
}
\makeatother

\DeclareMathOperator*{\vecop}{vec}
\DeclareMathOperator*{\trace}{trace}

\newcommand{\vvec}{\mathsf{v}}
\newcommand{\tran}{\ensuremath{\mkern-1.5mu\mathsf{T}}}

\usepackage{color}
\usepackage{MnSymbol}
\usepackage[linesnumbered,ruled,vlined]{algorithm2e}

\usepackage{stackengine}
\newcommand\xrowht[2][0]{\addstackgap[.5\dimexpr#2\relax]{\vphantom{#1}}}

\usepackage{caption}

\input{generic}

\usepackage{pgfplots}			
\pgfplotsset{compat=1.3}		

\usepackage{multirow}
\newcommand\reallywidehat[1]{%
  \savestack{\tmpbox}{\stretchto{%
      \scaleto{%
        \scalerel*[\widthof{\ensuremath{#1}}]{\kern-.6pt\bigwedge\kern-.6pt}%
        {\rule[-\textheight/2]{1ex}{\textheight}}
      }{\textheight}%
    }{0.5ex}}%
  \stackon[1pt]{#1}{\tmpbox}%
}

\usepackage[ulem=normalem,truncate=fit,todonotes={textsize=tiny}]{changes}
\definechangesauthor[name={Jens Saak}, color=blue!70!black]{JS}

\definechangesauthor[name={Davide Palitta}, color=magenta]{DP}

\definechangesauthor[name={Ivica Nakic}, color=orange]{IN}

\definechangesauthor[name={Zoran Tomljanovic}, color=green]{ZT}

\newcommand{\RR}{{\mathbb{R}}}
\newcommand{\CC}{{\mathbb{C}}}
\newcommand{\LL}{{\mathbf{L}}}
\newcommand{\cL}{{\mathcal{L}}}

\newcommand{\lyap}{{\ttfamily lyap}}

\DeclareOldFontCommand{\rm}{\normalfont\rmfamily}{\mathrm}
\DeclareOldFontCommand{\sf}{\normalfont\sffamily}{\mathsf}
\DeclareOldFontCommand{\tt}{\normalfont\ttfamily}{\mathtt}
\DeclareOldFontCommand{\bf}{\normalfont\bfseries}{\mathbf}
\DeclareOldFontCommand{\it}{\normalfont\itshape}{\mathit}     
\DeclareOldFontCommand{\sl}{\normalfont\slshape}{\@nomath\sl} 
\DeclareOldFontCommand{\sc}{\normalfont\scshape}{\@nomath\sc} 

\newcommand{\REV}[1]{\textcolor{black}{#1}}

\newtheorem{Prop}{Proposition}[section]
\newtheorem{Ass}{Assumption}

\theoremstyle{definition} 
\newtheorem{example}{Example}[section]

\begin{document}
\title{Efficient solution of sequences of parametrized Lyapunov equations with applications}

\author[$\ast$]{Davide~Palitta}
\affil[$\ast$]{%
  Universit\`a di Bologna,
  Centro $AM^2$,
  Dipartimento di Matematica,
  Piazza di Porta S. Donato 5,
  40127 Bologna,
  Italy.
  \authorcr%
  \email{davide.palitta@unibo.it},
  \orcid{0000-0002-6987-4430}
}

\author[$\dagger$]{Zoran~Tomljanovi\'{c}}
\affil[$\dagger$]{%
  School of Applied Mathematics and Informatics,
  J. J. Strossmayer University of Osijek,
  Trg Ljudevita Gaja 6,
  31000 Osijek,
  Croatia.
  \authorcr
  \email{ztomljan@mathos.hr},
  \orcid{0000-0002-3239-760X}
}

\author[$\ddag$]{Ivica~Naki\'{c}}
\affil[$\ddag$]{%
  Department of Mathematics,
  Faculty of Science,
  University of Zagreb,
  Bijeni\v{c}ka 30,
  10000 Zagreb,
  Croatia.
  \authorcr
  \email{nakic@math.hr},
  \orcid{0000-0001-6549-7220}
}

\author[$\S$]{Jens~Saak}
\affil[$\S$]{%
  Computational Methods in Systems and Control Theory (CSC),
  Max Planck Institute for Dynamics of Complex Technical Systems,
  Sandtorstra\ss{e} 1,
  39106 Magdeburg,
  Germany.\authorcr
  \email{saak@mpi-magdeburg.mpg.de},
  \orcid{0000-0001-5567-9637}
}

\shorttitle{Solving sequences of parametrized Lyapunov equations}
\shortauthor{D. Palitta, Z. Tomljanovi\'{c}, I. Naki\'{c}, J. Saak }


\msc{MSC1, MSC2, MSC3}

\abstract{Sequences of parametrized Lyapunov equations can be encountered in many application settings. Moreover, solutions of such equations are often intermediate steps of an overall procedure whose main goal is the computation of $\trace(EX)$ where $X$ denotes the solution of a Lyapunov equation and $E$ is a given matrix.
  We are interested in addressing problems where the parameter dependency of the coefficient matrix is encoded as a low-rank modification to a \emph{seed}, fixed matrix. We propose two novel numerical procedures that fully exploit such a common structure. The first one builds upon \REV{the Sherman-Morrison-Woodbury (SMW) formula and} recycling Krylov techniques, and it is well-suited for small dimensional problems as it makes use of dense numerical linear algebra tools. The second algorithm can instead address large-scale  problems by relying on state-of-the-art projection techniques based on the extended Krylov subspace.
  We test the new algorithms on several problems arising in the study of damped vibrational systems and the analyses of output synchronization problems for multi-agent systems. Our results show that the algorithms we propose are superior to state-of-the-art techniques as they are able to remarkably speed up the computation of accurate solutions.
}

\novelty{Novel, efficient methods for solving sequences of parametrized Lyapunov equations have been designed. By taking full advantage of the parameter dependency structure, we are able to tackle both small-scale and large-scale problems. The numerical results show the superiority of the proposed methods when compared to state-of-the-art routines, especially in terms of computational time.}

\maketitle

%
%


\section{Introduction and setting}\label{Introduction}

The main goal of this study is the design of efficient numerical procedures for solving sequences of parameter-dependent Lyapunov equations of the form
\begin{equation}
  \label{Ljap jdba 1}
  A(\vvec)X + X {A(\vvec)}^{\tran} = - Q,
\end{equation}
where $A(\vvec)$ has the form
\begin{equation}
  \label{eq:formA}
  A(\vvec) = A_0 - B_{l} D(\vvec) B_{r}^{\tran}.
\end{equation}
Here $A_0, Q \in \RR^{n\times n}$, $B_l,B_r \in \RR^{n\times k }$ are fixed matrices and the matrix $D(\vvec)=\diag{(v_1,v_2,\ldots,v_k)}$ contains parameters $v_{i}\in \RR \setminus \left\{ 0 \right\}$, for $i=1\ldots,k$, encoded in the parameter vector
$\vvec=\begin{bmatrix}v_1, v_2, \ldots,v_k\end{bmatrix}^{\tran}\in  {(\RR\setminus \left\{ 0 \right\})}^{k}$\,\footnote{For the sake of simplicity, we assume the data of our problem to be real.
Our algorithms can handle complex data by applying straightforward modifications.}. Note that, without loss of generality, we assume that the parameters in $\vvec$ are nonzero.  \REV{The number of parameters $k$ is supposed to be very small, $k=\mathcal{O}(1)$, as this is the case in the application settings we are interested in; see Section~\ref{Applications}.}

Structured equations of the form~\eqref{Ljap jdba 1} arise in many problem settings.  Indeed, let $b_l^i$, $b_r^i$, $i=1,\ldots,k$, be the columns of the matrices $B_l$ and $B_r$, respectively. Then $A(\vvec)= A_0 - \sum_{i=1}^k v_i b_l^i{(b_r^i)}^{\tran} =: A_0 + \sum_{i=1}^k v_i A_i$, and hence matrices of the form~\eqref{eq:formA} can describe a class of parameter-dependent matrices  with an affine dependency on the parameters. Lyapunov equations with such coefficient matrices naturally arise in the study of vibrational systems~\cite{TRUH04}, parametric model reduction~\cite{son2021balanced}, in the design of low gain feedback~\cite{zhou2008parametric}, and the analysis of multi-agent systems~\cite{su2013semi,Kim11}.
The applications we are going to study more closely are optimal damping of vibrational systems and  output-synchronization problems for multi-agent systems.

\REV{We are especially interested in devising efficient procedures for calculating $\trace(EX(\vvec))$, for some matrix $E$, for many different parameter vectors $\vvec$.
Indeed, $\trace(EX(\vvec))$ often corresponds to the target value we are interested in computing, with calculating $X(\vvec)$ only being an intermediate step. For instance this is the case when computing the $H_2$-norm of a linear system. An efficient computation of $\trace(EX(\vvec))$ is crucial for obtaining competitive numerical schemes in terms of both computational performance and memory requirements for the applications mentioned above. Our method also covers efficient calculation of quantities of the form $ f(X(\vvec)) $, for a given function $f$ that is assumed to be \emph{additive}, i.e.\ $f(X+Y)=f(X)+f(Y)$. To simplify the notation, in the sequel we will denote $\trace(EX)$ by $f(X)$.}


Naturally, one can solve~\eqref{Ljap jdba 1} from scratch for any $\vvec$. For instance, if the problem dimension allows, one can apply the Bartels-Stewart algorithm~\cite{Bartels1972} to each instance of~\eqref{Ljap jdba 1}. However, this naive procedure does not exploit the attractive structure of~\eqref{Ljap jdba 1} and requires the computation of the Schur decomposition of $A(\vvec)$ any time $\vvec$ changes. This approach is not affordable in terms of computational cost if $\vvec$ varies a lot and $f(X(\vvec))$ needs to be evaluated many times.
More sophisticated schemes for~\eqref{Ljap jdba 1} can be found in the literature. For instance, the algorithms presented in~\cite{TRUH04,KT13} do exploit the structure of $A(\vvec)$ in~\eqref{eq:formA}. Even though these schemes are more performing than a naive application of the Bartels-Stewart algorithm, they are not designed to capitalize on a possible slow variation in $\vvec$ as it often happens within, e.g., an optimization routine. Moreover, they can be applied to small dimensional problems only.
Various approaches for solving the parameter-dependent Lyapunov equations can be
found in the literature, including those based on the conjugate gradient
method~\cite{kressner2014preconditioned}, low-rank updates~\cite{Kressner2019},
interpolation on manifolds~\cite{journee2010low}, and low-rank reduced basis
method~\cite{son2017solving, morPrzPB23}.
On the other hand, they are not well-suited for solving sequences of numerous parameter-dependent equations.

Leveraging the structure of the coefficient matrix $A(\vvec)$ --- $A(\vvec)$ is an affine function in $\vvec$ --- we propose two different schemes which are able to fully separate the parameter-independent calculations from the $\vvec$-dependent computations. The former will be performed once and for all in an~\emph{offline} step, whereas the latter operations take place \emph{online}, namely every time $\vvec$ changes.

The first procedure we suggest in this paper builds upon the work in~\cite{TRUH04,KT13}, which we enhance with a recycling Krylov technique. This routine makes use of dense linear algebra tools so that it is well suited for small dimensional problems, i.e., for moderate values of $n$. The second routine addresses the large-scale case employing projection methods for linear matrix equations.

Both our algorithms take inspiration from the low-rank update scheme presented in~\cite{Kressner2019}, for which the following assumption is needed.

\begin{Ass}\label{ass:2}
  The solution $X_0$ to the Lyapunov equation $A_0X + X A_0^{\tran} = -Q $ can be computed efficiently.
\end{Ass}

\paragraph{Synopsis of the paper:} In Section~\ref{Efficient trace computation}
we present the main contribution of this work. In particular, a \REV{novel scheme combining the SMW formula with recycling
Krylov-like methods is illustrated in Section~\ref{Small-scale setting: recycling Krylov
  methods}.}  As already mentioned, this scheme is able to efficiently deal with small dimensional
problems only as it employs dense numerical linear
algebra tools like, e.g., the full eigenvalue decomposition of $A_0$, so
that it cannot be applied in the large-scale setting. We address the latter
scenario in Section~\ref{Large-scale setting: a projection framework}, where a
sophisticated projection technique, based on the extended Krylov
subspace~\eqref{eq:EK_def}, is illustrated.  \REV{In this framework, we will assume to be able to efficiently solve linear systems with $A_0$.} Some details about two application
settings where equations of the form~\eqref{Ljap jdba 1} are encountered are
reported in Section~\ref{Applications}. In particular, damped vibrational
systems are treated in
Section~\ref{subsec:application-damped_vibrational_systems}, whereas in
Section~\ref{subsec:multi-agent systems} multi-agent systems are
considered.
A panel of diverse numerical experiments displaying the effectiveness of our methodology is presented in Section~\ref{Numerical examples}, while Section~\ref{Conclusions} collects our conclusions.

Throughout the paper, we adopt the following notation. The symbol $\otimes$ denotes the Kronecker product and $\mathop{\circ}$ denotes the Hadamard (component-wise) product. Capital letters, both in roman ($A$) and italic ($\mathcal{A}$), denote matrices with no particular structure, whereas capital bold letters ($\mathbf{A}$) denote matrices having a Kronecker structure. Given a matrix $X\in\mathbb{R}^{m\times n}$, $\text{vec}(X)\in\mathbb{R}^{mn}$ denotes the vector obtained by stacking the columns of $X$ on top of each other. The matrix $I_n$ is the $n\times n$ identity matrix; the subscript is omitted whenever the dimension of $I$ is clear from the context. The symbol $e_i$ denotes the $i$-th canonical basis vector of $\mathbb{R}^s$ where $s$ is either specified in the text or clear from the context. The symbol $\delta_{i,j}$ denotes the Kronecker delta, i.e., $\delta_{i,j}=1$ if $i=j$ and zero otherwise.


\section{The novel solution methods}\label{Efficient trace computation}

Equation~\eqref{Ljap jdba 1} can be written as
\[
  (A_0-B_{l} D(\vvec) B_{r}^{\tran})X(\vvec)+X(\vvec)
  {(A_0- B_{l} D(\vvec) B_{r}^{\tran})}^{\tran}= -Q.
\]
By exploiting the low rank of $B_{l} D(\vvec) B_{r}^{\tran}$ and using the low-rank update approach presented in~\cite{Kressner2019}, we can also rewrite \(X\) as the sum of a fixed matrix and a low-rank update term, i.e.,
\begin{equation}\label{eq:splitXv}
  X(\vvec)=X_0+X_{\delta}(\vvec),
\end{equation}
where $X_0$ and $X_{\delta}(\vvec)$ are as follows. The matrix $X_0$ is independent of the parameter vector $\vvec$, and it solves the Lyapunov equation
\begin{equation}\label{eq:seedLE}
  A_0X_0+X_0A_0^{\tran}=- Q,
\end{equation}
whereas $X_{\delta}(\vvec)$ is such that
\begin{equation}\label{eq:Lyap_deltaX}
  A(\vvec)X_{\delta}(\vvec)+X_{\delta}(\vvec) {A(\vvec)}^{\tran}=B_{l} D(\vvec) B_{r}^{\tran} X_0+X_0 B_{r} D(\vvec) B_{l}^{\tran}.
\end{equation}
Assumption~\ref{ass:2} implies that $X_0$ can be efficiently calculated.

We notice that the right-hand side in~\eqref{eq:Lyap_deltaX} can be written as
\begin{equation}\label{eq:rhs_newexpression}
  B_{l} D(\vvec) B_{r}^{\tran} X_0+X_0 B_{r} D(\vvec) B_{l}^{\tran}=[X_0 B_{r},B_{l}]\begin{bmatrix}                                                                              0  & D(\vvec)\\                                                                            D(\vvec) & 0 \\                                                                                           \end{bmatrix}{[X_0 B_{r},B_{l}]}^{\tran}= P \left( \mathcal{I}_2\otimes D(\vvec)\right)
  P^{\tran},
\end{equation}
where $P=[X_0 B_{r},B_{l}]$ has rank $2k$, and $\mathcal{I}_2=\begin{bmatrix}
  0&1\\
  1& 0\\
\end{bmatrix}$.
This formulation of the right-hand side will be crucial for the projection-based method we design for large-scale problems in Section~\ref{Large-scale setting: a projection framework}.

In the next sections we show how to efficiently compute $X_{\delta}(\vvec)$. In particular, it turns out that many of the required operations do not depend on $\vvec$ so that they can be carried out in the offline stage.
Therefore, we split the solution process into offline and online phases, where only the operations that actually depend on $\vvec$ are performed online.

\subsection{Small-scale setting: SMW and recycling Krylov methods}\label{Small-scale setting: recycling Krylov methods}
In this section we address the numerical solution of the Lyapunov equation~\eqref{eq:Lyap_deltaX} assuming the problem dimension $n$ to be small, namely it is such that dense linear algebra operations costing up to $\mathcal O(n^3)$ flops are feasible.

Even though there are many efficient methods for solving Lyapunov equations of
moderate dimensions, see, e.g.,~\cite{Koe21} and references therein,  we want to efficiently solve~\eqref{eq:Lyap_deltaX} for a sequence of parameter vectors $\vvec$.

The strategy proposed in this section builds upon the work in~\cite{TRUH04,KT13}. We extend their approach to the more general class of problems of the form~\eqref{Ljap jdba 1}, and we allow for more than one parameter, \REV{i.e., $k$ can be larger than one}. The scheme we are going to present relies on the computation of the full eigenvalue decomposition of $A_0$, namely $A_0 = Q_0 \Lambda_0 Q_0^{-1}$ with $Q_0\in \CC^{n\times n}$ and $\Lambda_0 = \diag (\lambda_1,\ldots, \lambda_n)$, $\lambda_i\in \CC$, $i=1,\ldots,n$.



Considering the notation introduced in~\eqref{eq:rhs_newexpression}, the Kronecker form of~\eqref{eq:Lyap_deltaX} reads as follows
\begin{equation*}
  \left(I\otimes A_0+A_0\otimes I-
    I\otimes
    \left(B_{l} D(\vvec) B_{r}^{\tran}\right)-\left(B_{l} D(\vvec) B_{r}^{\tran}\right)\otimes I\right)\vecop(X_{\delta}(\vvec))=\left(P\otimes P\right)\vecop\left( \mathcal{I}_2\otimes D(\vvec)\right).
\end{equation*}

If  $A_0=Q_0\Lambda_0Q_0^{-1}$, $\Lambda_0=\mathop{\mathrm{diag}}(\lambda_1,\ldots,\lambda_{n})$,
we can write
\begin{multline*}
  (Q_0\otimes Q_0)\left(\LL_0-
    I\otimes
    \left(Q_0^{-1}B_{l} D(\vvec) B_{r}^{\tran}Q_0\right)-\left(Q_0^{-1}B_{l} D(\vvec) B_{r}^{\tran}Q_0\right)\otimes I\right)(Q_0^{-1}\otimes Q_0^{-1})\vecop(X_{\delta}(\vvec))=\\\left(P\otimes P\right)\vecop\left( \mathcal{I}_2\otimes D(\vvec)\right),
\end{multline*}
where  $\LL_0:=I\otimes \Lambda_0+\Lambda_0\otimes I$, so that we obtain
\begin{multline*}
  \left(\LL_0-
    \left[Q_0^{-1}B_{l} \otimes I, I\otimes Q_0^{-1}B_{l} \right]\mathbf{D}(\vvec){\left[Q_0^{\tran}B_{r}\otimes I,I\otimes Q_0^{\tran} B_{r}\right]}^{\tran}\right)\vecop(Q_0^{-1}X_{\delta}(\vvec)Q_0^{-\tran})=\\\left(Q_0^{-1}P\otimes Q_0^{-1}P\right)\vecop\left(\mathcal{I}_2\otimes D(\vvec) \right),
\end{multline*}
with $\mathbf{D}(\vvec)=\begin{bmatrix}                          D(\vvec)\otimes I & \\
  & I\otimes D(\vvec)\\
\end{bmatrix}$.

By denoting $\mathbf{M}:=\left[Q_0^{-1}B_{l} \otimes I,\, I\otimes Q_0^{-1}B_{l} \right]$, and $\mathbf{N}:=\left[Q_0^{\tran}B_{r}\otimes I,\,I\otimes Q_0^{\tran} B_{r}\right]$,  \\ the  Sherman–Morrison–Woodbury formula (see, e.g.,~\cite{Golub2013}) implies that
\begin{align}\label{eq:deltaX_SMW2}
  \vecop(X_{\delta}(\vvec))=&(Q_0\otimes Q_0)\bigg(\LL_0^{-1}\left(Q_0^{-1}P\otimes Q_0^{-1}P\right)\vecop\left(\mathcal{I}_2\otimes D(\vvec) \right)+\notag\\
                            &\LL_0^{-1} \mathbf{M}{\left({\mathbf{D}(\vvec)}^{-1}- \mathbf{N}^{\tran}\LL_0^{-1} \mathbf{M}\right)}^{-1}\mathbf{N}^{\tran}\LL_0^{-1}\left(Q_0^{-1}P\otimes Q_0^{-1}P\right)\vecop\left(\mathcal{I}_2\otimes D(\vvec)\right)\bigg).
\end{align}

\REV{The use of the Sherman-Morrison-Woodbury formula in the matrix setting is not new; see, e.g.,~\cite{HaoSim2021,Kuz2013},~\cite[Section 3]{Damm2008},~\cite[Section 5]{Massetal2018},~\cite{BenLP08,Pen00a}. However, applying the general-purpose schemes presented in these papers to our parameter-dependent framework would not lead to an efficient solution process. Indeed, this would be equivalent to explicitly computing~\eqref{eq:deltaX_SMW2} every time $\vvec$ changes. Our} main goal here is to divide the operations that depend on $\vvec$ from those that do not, as much as possible, such that the \(\vvec\)-independent calculations can be performed offline, \REV{only once. In the interest of the reader, we thus spell out the details of our procedure, in spite of some similarities with prior work.}

We first focus on the term $(Q_0\otimes Q_0)\left(\LL_0^{-1}\left(Q_0^{-1}P\otimes Q_0^{-1}P\right)\vecop\left(\mathcal{I}_2\otimes D(\vvec) \right)\right)$.
The computation of
\[
  z(\vvec)=(Q_0\otimes Q_0)\left(\LL_0^{-1}\left(Q_0^{-1}P\otimes
      Q_0^{-1}P\right)\vecop\left(\mathcal{I}_2\otimes D(\vvec)
    \right)\right),
\]
is equivalent to solving the Lyapunov equation
\[
  \Lambda_0 \widetilde Z(\vvec)+\widetilde Z(\vvec)\Lambda_0^{\tran}=Q_0^{-1}P\left(\mathcal{I}_2\otimes D(\vvec) \right)P^{\tran}Q_0^{-\tran},
\]
where \(z(\vvec)=\vecop(Z(\vvec))\), and \(Z(\vvec):= Q_0\widetilde Z(\vvec)Q_0^{\tran}\).
Since $\Lambda_0=\mathop{\mathrm{diag}}(\lambda_1,\ldots,\lambda_{n})$, by defining the Cauchy matrix $\cL\in\CC^{n\times n}$ whose $(i,j)$-th entry is given by $\cL_{i,j}=1/(\lambda_{i}+ \lambda_{j})$, we can write
\begin{equation}\label{eq:Z(v)}
  Z(\vvec)=Q_0\left(\cL\mathop{\circ}\left(Q_0^{-1}P\left(\mathcal{I}_2\otimes D(\vvec) \right)P^{\tran}Q_0^{-\tran}\right)\right)Q_0^{\tran}.
\end{equation}

The matrix $\mathcal{I}_2\otimes D(\vvec)$ can be written as
\[\mathcal{I}_2\otimes D(\vvec)=\sum_{i=1}^k v_{i}(e_{k+i}e_{i}^{\tran}+e_{i}e_{k+i}^{\tran}),\]
and by plugging this expression into~\eqref{eq:Z(v)}, we get
\begin{equation}\label{eq:Z(v)2}
  Z(\vvec)=\sum_{i=1}^{k}v_{i}Z_{i}=\sum_{i=1}^{k}v_{i}Q_{0}\widetilde Z_{i}Q_{0}^{\tran},\quad \widetilde Z_{i}:=\cL\mathop{\circ}\left(Q_{0}^{-1}P\left(e_{k+i}e_{i}^{\tran}+e_{i}e_{k+i}^{\tran}\right)P^{\tran}Q_0^{-\tran}\right).
\end{equation}
We can thus compute the matrices $Z_{i}$ before any change in the parameters $\vvec$ takes place. Whenever we need $Z(\vvec)$, we just compute the linear combination in~\eqref{eq:Z(v)2}.

The second term in~\eqref{eq:deltaX_SMW2} can be written as
\begin{multline}\label{eq:SMW_secondterm}
  \left(Q_0\otimes Q_0\right)\LL_0^{-1} \mathbf{M}{\left({\mathbf{D}(\vvec)}^{-1}-\mathbf{N}^{\tran}\LL_0^{-1} \mathbf{M}\right)}^{-1} \mathbf{N}^{\tran}\LL_0^{-1}\left(Q_0^{-1}P\otimes Q_0^{-1}P\right)\vecop\left(\mathcal{I}_2\otimes D(\vvec)\right)=\\
  \left(Q_0\otimes Q_0\right)
  \LL_0^{-1} \mathbf{M}{\left({\mathbf{D}(\vvec)}^{-1}-\mathbf{N}^{\tran}\LL_{0}^{-1} \mathbf{M}\right)}^{-1}\mathbf{N}^{\tran}\left(\sum_{i=1}^{k}v_{i}\vecop\left(\widetilde Z_{i}\right)\right).
\end{multline}
Recalling that $\mathbf{N}:=\left[Q_0^{\tran}B_{r}\otimes I,I\otimes Q_0^{\tran} B_{r}\right]$, we have
\[
  \mathbf{N}^{\tran}\vecop\left(\widetilde  Z_{i}\right)=
  \begin{bmatrix}
    \vecop\left(\widetilde  Z_{i} Q_0^{\tran}B_{r}\right)\\
    \vecop\left(B_{r}^{\tran}Q_0\widetilde  Z_{i}\right)\\
  \end{bmatrix}, \quad\text{so that} \quad
  \mathbf{N}^{\tran}\left(\sum_{i=1}^{k}v_{i}\vecop\left(\widetilde
      Z_{i}\right)\right)=\sum_{i=1}^{k}v_{i}
  \begin{bmatrix}
    \vecop\left(\widetilde  Z_{i} Q_{0}^{\tran}B_{r}\right)\\
    \vecop\left(B_{r}^{\tran}Q_0\widetilde  Z_{i}\right)\\
  \end{bmatrix}.
\]

We now focus on the solution of the $2nk\times 2nk$ linear system with ${\mathbf{D}(\vvec)}^{-1}-\mathbf{N}^{\tran}\LL_0^{-1} \mathbf{M}$. We first show how to efficiently assemble the coefficient matrix itself.

While ${\mathbf{D}(\vvec)}^{-1}=
\begin{bmatrix}
  {D(\vvec)}^{-1}\otimes I& \\
  &I\otimes {D(\vvec)}^{-1}\\
\end{bmatrix}$
is a diagonal matrix, the structure of $\mathbf{N}^{\tran}\LL_0^{-1} \mathbf{M}$ is more involved.
A first study of its structure can be found in~\cite{TRUH04}. Unlike what is done in~\cite{TRUH04}, here we exploit the diagonal pattern of $\LL_0^{-1}$, which leads to a useful representation of the blocks of $\mathbf{N}^{\tran}\LL_0^{-1} \mathbf{M}$ when the latter is viewed as a $2\times 2$ block matrix. Such a block structure will help us to design efficient preconditioners for the solution of the linear system in~\eqref{eq:SMW_secondterm} as well.

Let $ \mathbf{M}_1= Q_0^{-1}B_{l}\otimes I$, $ \mathbf{M}_2=I\otimes Q_0^{-1}B_{l} $, $\mathbf{N}_1=  Q_0^{\tran}B_{r}\otimes I$, and $\mathbf{N}_2=I \otimes  Q_0^{\tran}B_{r}$. Then we can write
\[
  \mathbf{N}^{\tran}\LL_0^{-1} \mathbf{M}=
  \begin{bmatrix}
    \mathbf{N}_1^{\tran}\LL_0^{-1} \mathbf{M}_1 & \mathbf{N}_1^{\tran}\LL_0^{-1} \mathbf{M}_2\\
    \mathbf{N}_2^{\tran}\LL_0^{-1} \mathbf{M}_1 & \mathbf{N}_2^{\tran}\LL_0^{-1}
    \mathbf{M}_2
  \end{bmatrix}.
\]
We first focus on the $(1,1)$-block, namely $\mathbf{N}_1^{\tran}\LL_0^{-1} \mathbf{M}_1\in\CC^{nk\times nk}$.
By recalling that the $h$-th basis vector of $\CC^{nk}$ can be written as $e_h=\vecop(e_{s}e_{t}^{\tran})$, with $e_{s}$ and $e_{t}$ being the canonical vectors in $\CC^{n}$ and $\CC^{k}$, respectively, and $h=(t-1)n+s$,  we can write the
$(i,j)$-th entry of $\mathbf{N}_1^{\tran}\LL_0^{-1} \mathbf{M}_1\in\CC^{nk\times nk}$ as
\begin{align*}
  e_{i}^{\tran}\mathbf{N}_{1}^{\tran}\LL_0^{-1}
  \mathbf{M}_{1}e_{j}=&{\vecop(e_{s}e_{t}^{\tran})}^{\tran}\mathbf{N}_{1}^{\tran}\LL_{0}^{-1}
                        \mathbf{M}_{1}\vecop(e_{q}e_{p}^{\tran}),
\end{align*}
where $i=(t-1)n+s$ and $j=(p-1)n+q$. Since $\mathbf{N}_1$ and $ \mathbf{M}_1$ have a Kronecker structure, we have
\begin{align*}
  e_{i}^{\tran}\mathbf{N}_{1}^{\tran}\LL_{0}^{-1} \mathbf{M}_{1}e_{j}=&{\vecop(e_{s}e_{t}^{\tran}B_{r}^{\tran}Q_0)}^{\tran}\LL_{0}^{-1}\vecop(e_{q}e_{p}^{\tran}B_{l}^{\tran}Q_{0}^{-\tran})= {\vecop(e_{s}e_{t}^{\tran}B_{r}^{\tran}Q_0)}^{\tran}\vecop(\cL\mathop{\circ}(e_{q}e_{p}^{\tran}B_{l}^{\tran}Q_{0}^{-\tran}))\\
  =&\trace{\left(Q_{0}^{\tran}B_{r} e_{t}e_{s}^{\tran}(\cL\mathop{\circ}(e_{q}e_{p}^{\tran}B_{l}^{\tran}Q_{0}^{-\tran}))\right)}=e_{s}^{\tran}
     (\cL\mathop{\circ}(e_{q}e_{p}^{\tran}B_{l}^{\tran}Q_{0}^{-\tran}))Q_{0}^{\tran}B_{r} e_{t},
\end{align*}
where the last equality holds thanks to the cyclic property of the trace operator. Now, by applying
the following property of the Hadamard product
\[x^{\tran}(A\mathop{\circ} B)y=\trace{(\mathop{\mathrm{diag}}(x)A\mathop{\mathrm{diag}}(y)B^{\tran})},\]
we get
\begin{align}\label{eq:N1M1}
  e_{i}^{\tran}\mathbf{N}_1^{\tran}\LL_{0}^{-1} \mathbf{M}_{1}e_{j}=&\trace{\left(\mathop{\mathrm{diag}}(e_s)\cL\mathop{\mathrm{diag}}(Q_0^{\tran}B_{r} e_{t})
                                                                      Q_0^{-1}B_{l} e_{p}e_{q}^{\tran}\right)}=e_{q}^{\tran}\mathop{\mathrm{diag}}(e_{s})\cL\mathop{\mathrm{diag}}(Q_{0}^{\tran}B_{r} e_{t})
                                                                      Q_0^{-1}B_{l} e_p\notag\\
  =& e_{q}^{\tran}e_{s}e_{s}^{\tran}\cL\mathop{\mathrm{diag}}(Q_{0}^{\tran}B_{r} e_{t})
     Q_0^{-1}B_{l} e_{p}=
     \delta_{q,s}\cdot e_{s}^{\tran}\cL\mathop{\mathrm{diag}}(Q_0^{\tran}B_{r} e_{t})
     Q_0^{-1}B_{l} e_p\notag\\
  =& \delta_{q,s}\cdot e_{s}^{\tran}\left(\cL\left((Q_0^{\tran}B_{r} e_t)\mathop{\circ} (Q_0^{-1}B_{l} e_p)\right)\right),
\end{align}
where $\delta_{q,s}$ denotes the Kronecker delta. The relation in~\eqref{eq:N1M1} shows that $\mathbf{N}_1^{\tran}\LL_0^{-1} \mathbf{M}_1$ is a $k\times k$ block matrix whose blocks are all diagonal. In particular, the block in the $(t,p)$ position is given by the matrix $\mathop{\mathrm{diag}}\left(\cL\left((Q_0^{\tran}B_{r} e_t)\mathop{\circ} (
    Q_0^{-1}B_{l} e_p)\right)\right)$.

We now derive the structure of the $(2,2)$-block $\mathbf{N}_2^{\tran}\LL_0^{-1} \mathbf{M}_2$. Following the same reasoning as before, we have
\begin{align*}
  e_{i}^{\tran}\mathbf{N}_2^{\tran}\LL_{0}^{-1} \mathbf{M}_{2}e_{j}=&{\vecop(Q_{0}^{\tran}B_{r} e_{t}e_{s}^{\tran})}^{\tran}\LL_{0}^{-1}\vecop(Q_{0}^{-1}B_{l} e_{p}e_{q}^{\tran})= {\vecop(Q_{0}^{\tran}B_{r} e_{t}e_{s}^{\tran})}^{\tran}\vecop(\cL\mathop{\circ}(Q_0^{-1}B_{l} e_{p}e_{q}^{\tran}))\\
  =&\trace{\left(e_s e_t^{\tran}B_{r} ^{\tran}Q_0
     (\cL\mathop{\circ}(Q_0^{-1}B_{l} e_{p}e_{q}^{\tran}))\right)}=e_t^{\tran}B_{r} ^{\tran}Q_0
     (\cL\mathop{\circ}(Q_0^{-1}B_{l} e_{p}e_{q}^{\tran})) e_s.
\end{align*}
Notice that in the expression above we wrote $e_{i}=\vecop(e_{t}e_{s}^{\tran})$, with $e_{t}$, $e_{s}$ canonical vectors in $\CC^{k}$ and $\CC^{n}$ respectively, so that $i=(s-1)k+t$. Similarly, $j=(q-1)k+p$. As before, we can write
\begin{align*}
  e_{i}^{\tran}\mathbf{N}_2^{\tran}\LL_0^{-1} \mathbf{M}_2e_{j}=&\trace{\left(\mathop{\mathrm{diag}}(Q_0^{\tran}B_{r} e_t)
                                                                  \cL\mathop{\mathrm{diag}}(e_s)e_q
                                                                  e_p^{\tran}B_{l}^{\tran}Q_0^{-\tran}\right)}=e_p^{\tran}B_{l}^{\tran}Q_0^{-\tran}\mathop{\mathrm{diag}}(Q_0^{\tran}B_{r} e_t)
                                                                  \cL e_{s}e_{s}^{\tran} e_q\\
  =&\delta_{s,q}\cdot e_p^{\tran}B_{l}^{\tran}Q_0^{-\tran}\mathop{\mathrm{diag}}(Q_0^{\tran}B_{r} e_t)
     \cL e_s= \delta_{s,q}\cdot{\left((Q_0^{-1} B_{l} e_p)\mathop{\circ}(Q_0^{\tran}B_{r} e_t)
     \right)}^{\tran}\cL e_s.
\end{align*}
Therefore, $\mathbf{N}_2^{\tran}\LL_0^{-1} \mathbf{M}_2$ is a block diagonal matrix with $n$ diagonal blocks of order $k$. In particular, the $(t,p)$-th entry of the $s$-th diagonal block is given by
${\left((Q_0^{-1} B_{l} e_p)\mathop{\circ}(Q_0^{\tran}B_{r} e_t)
  \right)}^{\tran}\cL e_s$.

Unlike its diagonal blocks, the off-diagonal blocks of $\mathbf{N}^{\tran}\LL_0^{-1} \mathbf{M}$, namely $\mathbf{N}_2^{\tran}\LL_0^{-1} \mathbf{M}_1$, and $\mathbf{N}_1^{\tran}\LL_0^{-1} \mathbf{M}_2$, do not possess a structured sparsity pattern. Nevertheless, we can still apply the same strategy we employed above to construct $\mathbf{N}_2^{\tran}\LL_0^{-1} \mathbf{M}_1$ and $\mathbf{N}_1^{\tran}\LL_0^{-1} \mathbf{M}_2$ while avoiding the explicit computation of $\mathbf{M}_1$, $\mathbf{M}_2$, $\mathbf{N}_1$, and $\mathbf{N}_2$. If $i=(s-1)k+t$ and $j=(p-1)n+q$, we have
\begin{align}\label{eq:N2M1}
  e_{i}^{\tran}\mathbf{N}_2^{\tran}\LL_0^{-1} \mathbf{M}_{1}e_{j}
  =&\vecop{(Q_0^{\tran}B_{r}
     e_{t}e_{s}^{\tran})}^{\tran}\LL_0^{-1}\vecop(e_{q}e_{p}^{\tran}B_{l}^{\tran}Q_0^{-\tran})
     = \vecop{(Q_0^{\tran}B_{r} e_{t}e_{s}^{\tran})}^{\tran}\vecop(\cL\mathop{\circ}(e_{q}e_{p}^{\tran}B_{l}^{\tran}Q_0^{-\tran}))\notag\\
  =&\trace{\left(e_{s}e_{t}^{\tran}B_{r}^{\tran}Q_0
     (\cL\mathop{\circ}(e_{q}e_{p}^{\tran}B_{l}^{\tran}Q_0^{-\tran}))\right)} = e_{t}^{\tran}B_{r}^{\tran}Q_0
     (\cL\mathop{\circ}(e_{q}e_{p}^{\tran}B_{l}^{\tran}Q_0^{-\tran})) e_{s}\notag\\
  =& \trace{\left(\mathop{\mathrm{diag}}(Q_{0}^{\tran}B_{r} e_{t})\cL\mathop{\mathrm{diag}}(e_{s})Q_{0}^{-1}B_{l} e_{p}e_{q}^{\tran}\right)}=
     e_q^{\tran}\mathop{\mathrm{diag}}(Q_0^{\tran}B_{r} e_t)\cL\mathop{\mathrm{diag}}(e_s)Q_0^{-1}B_{l} e_p\notag\\
  =& (e_q^{\tran}Q_0^{\tran}B_{r} e_t)(e_q^{\tran}\cL e_s)(
     e_s^{\tran}Q_0^{-1}B_{l} e_p).
\end{align}
Similarly, if $i=(t-1)n+s$ and $j=(q-1)k+p$, it holds
\begin{align}\label{eq:N1M2}
  e_{i}^{\tran}\mathbf{N}_1^{\tran}\LL_0^{-1} \mathbf{M}_2e_{j}
  =&{\vecop(e_{s}e_{t}^{\tran}B_{r}^{\tran}Q_0)}^{\tran}\LL_0^{-1}\vecop(Q_0^{-1}B_{l}e_{p}e_{q}^{\tran})
     = {\vecop(e_{s}e_{t}^{\tran}B_{r}^{\tran}Q_0)}^{\tran}\vecop(\cL\mathop{\circ}(Q_0^{-1}B_{l} e_{p}e_{q}^{\tran}))\notag\\
  =&\trace{\left(Q_0^{\tran}B_{r} e_{t}e_{s}^{\tran}
     (\cL\mathop{\circ}(Q_0^{-1}B_{l} e_{p}e_{q}^{\tran}))\right)} = e_{s}^{\tran}
     (\cL\mathop{\circ}(Q_0^{-1}B_{l} e_{p}e_{q}^{\tran})) Q_0^{\tran}B_{r} e_t\notag\\
  =& \trace{\left(\mathop{\mathrm{diag}}(e_s)\cL\mathop{\mathrm{diag}}(Q_0^{\tran}B_{r} e_t)e_{q}e_{p}^{\tran}B_{l}^{\tran} Q_0^{-\tran} \right)}=
     e_p^{\tran}B_{l}^{\tran} Q_0^{-\tran}\mathop{\mathrm{diag}}(e_s)\cL\mathop{\mathrm{diag}}(Q_0^{\tran}B_{r} e_t)e_q\notag\\
  =& (e_p^{\tran}B_{l}^{\tran} Q_0^{-\tran} e_s)(e_s^{\tran}\cL e_q)(
     e_q^{\tran}Q_0^{\tran}B_{r} e_t).
\end{align}

The construction of $\mathbf{N}^{\tran}\LL_0^{-1}\mathbf{M}$ can be carried out before starting changing $\vvec$. Once this task is accomplished, we have to solve a $2nk\times 2nk$ linear system with ${\mathbf{D}(\vvec)}^{-1}-\mathbf{N}^{\tran}\LL_0^{-1}\mathbf{M}$ every time we have to solve~\eqref{Ljap jdba 1} for a different $\vvec$.

We want to efficiently compute the vector $w(\vvec)\in\CC^{2nk}$ such that
\begin{equation}\label{eq:SMW_linearsystem}
  \left({\mathbf{D}(\vvec)}^{-1}-\mathbf{N}^{\tran}\LL_0^{-1}\mathbf{M}\right)w(\vvec)=\sum_{i=1}^{k}v_{i}
  \begin{bmatrix}
    \vecop\left(\widetilde  Z_{i} Q_0^{\tran}B_{r}\right)\\
    \vecop\left(B_{r}^{\tran}Q_0\widetilde  Z_{i}\right)\\
  \end{bmatrix}.
\end{equation}
We assume the coefficient matrix and the right-hand side in~\eqref{eq:SMW_linearsystem} to slowly change with $\vvec$.
For instance, if~\eqref{Ljap jdba 1} is embedded in a parameter optimization procedure, the parameters $\vvec_{\ell}={[v_1^{(\ell)},\ldots,v_k^{(\ell)}]}^{\tran}$ in one step do not dramatically differ from the ones used in the previous step, namely
$\vvec_{\ell-1}={[v_1^{(\ell-1)},\ldots,v_k^{(\ell-1)}]}^{\tran}$, especially in the later steps of the adopted optimization procedure.
For a sequence of linear systems of this nature, recycling Krylov techniques represent one of the most valid families of solvers available in the literature. See, e.g.,~\cite{Parks2006,Gaul2014} and the recent survey paper~\cite{Soodhalter2020}.

We employ the GCRO-DR method proposed in~\cite{Parks2006} to solve the sequence of linear systems in~\eqref{eq:SMW_linearsystem}. See~\cite{Parks2016a,Parks2016} for a Matlab implementation and a thorough discussion on certain computational aspects of this algorithm.

Once the $(\ell-1)$-th linear system is solved, GCRO-DR uses the space spanned by $s_{\ell}$ approximate eigenvectors of ${\mathbf{D}(\vvec_{\ell-1})}^{-1}-\mathbf{N}^{\tran}\LL_0^{-1}\mathbf{M}$
to enhance the solution of the $\ell$-th linear system
\[
  \left({\mathbf{D}(\vvec_\ell)}^{-1}-\mathbf{N}^{\tran}\LL_{0}^{-1}\mathbf{M}\right)w(\vvec_\ell)=\sum_{i=1}^{k}v_{i}^{(\ell)}
  \begin{bmatrix}
    \vecop\left(\widetilde  Z_{i} Q_0^{\tran}B_{r}\right)\\
    \vecop\left(B_{r}^{\tran}Q_0\widetilde  Z_{i}\right)\\
  \end{bmatrix}.
\]
We employ the $s_\ell$ eigenvectors corresponding to the $s_\ell$ eigenvalues of smallest magnitude\footnote{This is the default strategy in~\cite{Parks2016a}.}, but different approximate eigenvectors can be used as well; see~\cite[Section 2.4]{Parks2006}.

\REV{The convergence of GCRO-DR is driven by the spectral properties of ${\mathbf{D}(\vvec)}^{-1}-\mathbf{N}^{\tran}\LL_0^{-1}\mathbf{M}$.
However, the latter ones are extremely tricky to identify in general. Only by allowing further assumptions on $B_l$, $B_r$, and $A_0$, we may come up with sensible statements. For instance, if $B_l=B_r$ and $A_0$ is Hermitian, ${\mathbf{D}(\vvec)}^{-1}-\mathbf{N}^{\tran}\LL_0^{-1}\mathbf{M}$ is Hermitian too and GCRO-DR can take advantage of that; see~\cite[Theorem 3.1]{Parks2006}. Similarly, if, in addition, all the parameters $v_i$'s are, e.g., strictly positive and $A_0$ is stable, then ${\mathbf{D}(\vvec)}^{-1}-\mathbf{N}^{\tran}\LL_0^{-1}\mathbf{M}$ is also positive definite. GCRO-DR can take advantage of this property as well.
}

GCRO-DR allows for preconditioning.
In particular, thanks to the pattern of ${\mathbf{D}(\vvec)}^{-1}-\mathbf{N}^{\tran}\LL_0^{-1}\mathbf{M}$ we are able to design efficient preconditioning operators, which can be cheaply tuned to address the variation in the parameters.

We have
\[
  {\mathbf{D}(\vvec)}^{-1}-\mathbf{N}^{\tran}\LL_0^{-1}\mathbf{M} =
  \begin{bmatrix}
    {D(\vvec)}^{-1}\otimes I - \mathbf{N}_1^{\tran}\LL_0^{-1} \mathbf{M}_1 & \mathbf{N}_1^{\tran}\LL_0^{-1} \mathbf{M}_2\\
    \mathbf{N}_2^{\tran}\LL_0^{-1} \mathbf{M}_1 & I\otimes {D(\vvec)}^{-1} -
    \mathbf{N}_2^{\tran}\LL_0^{-1}
    \mathbf{M}_2\\
  \end{bmatrix},
\]
and we want to maintain the $(2\times 2)$-block structure in the preconditioner as well.

In ${\mathbf{D}(\vvec)}^{-1}-\mathbf{N}^{\tran}\LL_0^{-1}\mathbf{M}$, only the diagonal blocks ${D(\vvec)}^{-1}\otimes I - \mathbf{N}_1^{\tran}\LL_0^{-1} \mathbf{M}_1$ and
$I\otimes {D(\vvec)}^{-1} - \mathbf{N}_2^{\tran}\LL_0^{-1} \mathbf{M}_2$ depend on the current parameters. Moreover, both ${D(\vvec)}^{-1}\otimes I$ and $I\otimes {D(\vvec)}^{-1}$ are diagonal matrices so that the sparsity pattern of $\mathbf{N}_1^{\tran}\LL_0^{-1} \mathbf{M}_1$ and $\mathbf{N}_2^{\tran}\LL_0^{-1} \mathbf{M}_2$ is preserved. In particular, ${D(\vvec)}^{-1}\otimes I - \mathbf{N}_1^{\tran}\LL_0^{-1} \mathbf{M}_1$ is still a $k\times k$ block matrix with diagonal blocks whereas $I\otimes {D(\vvec)}^{-1} - \mathbf{N}_2^{\tran}\LL_0^{-1} \mathbf{M}_2$ is block-diagonal with $k\times k$ blocks on the diagonal.
By exploiting such a significant sparsity pattern, we are thus able to
efficiently solve linear systems
with ${D(\vvec)}^{-1}\otimes I - \mathbf{N}_1^{\tran}\LL_{0}^{-1} \mathbf{M}_1$ and $I\otimes {D(\vvec)}^{-1} - \mathbf{N}_2^{\tran}\LL_0^{-1} \mathbf{M}_2$ by a sparse direct method, regardless of the change in~$\vvec$.

The off-diagonal blocks of ${\mathbf{D}(\vvec)}^{-1}-\mathbf{N}^{\tran}\LL_0^{-1}\mathbf{M}$, namely $\mathbf{N}_1^{\tran}\LL_0^{-1} \mathbf{M}_2$ and $\mathbf{N}_2^{\tran}\LL_0^{-1} \mathbf{M}_1$,  are dense in general. We approximate these blocks by means of their truncated SVD (TSVD) of order $\widebar p$, for a user-defined $\widebar p>0$. Notice that these TSVDs can be computed once and for all, since neither $\mathbf{N}_1^{\tran}\LL_0^{-1} \mathbf{M}_2$ nor $\mathbf{N}_2^{\tran}\LL_0^{-1} \mathbf{M}_1$ depend on $\vvec$.
We denote the results of the TSVD by $\mathcal N_1 {\mathcal M_2}^{\tran}$ and $\mathcal N_2 {\mathcal M_1}^{\tran}$ where $\mathcal N_{i}$, $\mathcal M_{i}\in\CC^{nk\times \widebar p}$, $i=1,2$, and
$\mathcal N_1 {\mathcal M_2}^{\tran}\approx \mathbf{N}_1^{\tran}\LL_0^{-1} \mathbf{M}_2$ whereas
$\mathcal N_2 {\mathcal M_1}^{\tran}\approx \mathbf{N}_2^{\tran}\LL_0^{-1} \mathbf{M}_1$.

Consequently, the preconditioning operator we employ is
\begin{equation}\label{eq:prec}
  \mathcal{P}(\vvec)=
  \begin{bmatrix}
    {D(\vvec)}^{-1}\otimes I - \mathbf{N}_1^{\tran}\LL_0^{-1} \mathbf{M}_1 & \mathcal N_1 {\mathcal M_2}^{\tran}\\
    \mathcal{N}_2 {\mathcal{M}_1}^{\tran} & I\otimes {D(\vvec)}^{-1} -
    \mathbf{N}_2^{\tran}\LL_0^{-1} \mathbf{M}_2\\
  \end{bmatrix}\approx {\mathbf{D}(\vvec)}^{-1}-\mathbf{N}^{\tran}\LL_0^{-1}\mathbf{M}.
\end{equation}
We use right preconditioning within GCRO-DR so that at each  iteration a linear system with
$\mathcal{P}(\vvec)$ needs to be solved. However, such a task is particularly cheap. Indeed, thanks to the \(2 \times 2\) block structure of $\mathcal{P}(\vvec)$, we define its Schur complement

\[
  \mathcal{S}(\vvec)= {D(\vvec)}^{-1}\otimes I -
  \mathbf{N}_1^{\tran}\LL_0^{-1} \mathbf{M}_1 - \mathcal{N}_1 {\mathcal{M}_2}^{\tran}
  {\left(I\otimes {D(\vvec)}^{-1} - \mathbf{N}_2^{\tran}\LL_0^{-1} \mathbf{M}_2\right)}^{-1}\mathcal N_2 {\mathcal{M}_1}^{\tran},
\]
and, since $\mathcal{N}_1 {\mathcal{M}_2}^{\tran}{\left(I\otimes
    {D(\vvec)}^{-1} - \mathbf{N}_2^{\tran}\LL_0^{-1} \mathbf{M}_2\right)}^{-1}\mathcal{N}_2
{\mathcal{M}_1}^{\tran}$ has rank $\widebar{p}$, we can employ the Sherman-Morrison-Woodbury formula to efficiently solve linear systems with $\mathcal{S}(\vvec)$. We, thus, have
{\footnotesize
  \[
    {\mathcal{P}(\vvec)}^{-1}
    \begin{bmatrix}
      v_1\\ v_2\\
    \end{bmatrix}=
    \begin{bmatrix}
      {\mathcal{S}(\vvec)}^{-1}\left(v_1-\mathcal{N}_1\mathcal{M}_2^{\tran}{\left(I\otimes {D(\vvec)}^{-1} - \mathbf{N}_2^{\tran}\LL_0^{-1} \mathbf{M}_2\right)}^{-1}v_2\right)\\
      {\left(I\otimes {D(\vvec)}^{-1} - \mathbf{N}_2^{\tran}\LL_0^{-1} \mathbf{M}_2\right)}^{-1}\left(\mathcal{N}_2\mathcal{M}_1^{\tran}{\mathcal{S}(\vvec)}^{-1}\left(\mathcal{N}_1\mathcal{M}_2^{\tran}{\left(I\otimes {D(\vvec)}^{-1} - \mathbf{N}_2^{\tran}\LL_0^{-1} \mathbf{M}_2\right)}^{-1}v_2-v_1\right)+v_2
      \right)
    \end{bmatrix}.
  \]
}

Once the linear system in~\eqref{eq:SMW_linearsystem} is solved and the vector $w(\vvec)\in\CC^{2nk}$ is computed, we proceed with the remaining operations in~\eqref{eq:SMW_secondterm}. We can write
\begin{align*}
  \left(Q_0\otimes Q_0\right)
  \LL_0^{-1} \mathbf{M}w(\vvec)=&  \left(Q_0\otimes Q_0\right)
                                  \LL_0^{-1} \left[Q_0^{-1}B_{l} \otimes I, I\otimes Q_0^{-1}B_{l}\right]
                                  \begin{bmatrix}
                                    w_1(\vvec)
                                    \\
                                    w_2(\vvec)
                                    \\
                                  \end{bmatrix}\\
  =& \left(Q_0\otimes Q_0\right)
     \LL_0^{-1}\left(\vecop\left(W_1(\vvec)B_{l}^{\tran}Q_0^{-\tran}+
     Q_0^{-1}B_{l} W_2(\vvec)\right)\right)\\
  =& \vecop\left(Q_0 \left(\cL\mathop{\circ}\left(W_1(\vvec)B_{l}^{\tran}Q_0^{-\tran}+
     Q_0^{-1}B_{l} W_2(\vvec)\right)\right)  Q_0^{\tran}\right)=\vecop\left(W(\vvec) \right),
\end{align*}
where $W_1(\vvec)\in\CC^{n\times k}$ and $W_2(\vvec)\in\CC^{k\times n}$ are such that $\vecop(W_1(\vvec))=w_1(\vvec)$ and $\vecop(W_2(\vvec))=w_2(\vvec)$.

The solution $X(\vvec)$ to~\eqref{Ljap jdba 1}, in case of moderate $n$, can thus be computed as follows,
\begin{align}
  X(\vvec)=&X_0+X_{\delta}(\vvec)=X_0+Z(\vvec)+W(\vvec)\notag\\
  =& X_0+
     \sum_{i=1}^{k}v_{i}^{(\ell)}Z_{i}+Q_0 \left(\cL\mathop{\circ}\left(W_1(\vvec)B_{l}^{\tran}Q_0^{-\tran}+
     Q_0^{-1}B_{l} W_2(\vvec)\right)\right)  Q_0^{\tran},\label{eq:Xsmalllastline}
\end{align}
so that
\[
  f(X(\vvec))=f(X_0)+
  \sum_{i=1}^{k}v_{i}^{(\ell)}f(Z_{i})+f(Q_0 \left(\cL\mathop{\circ}\left(W_1(\vvec)B_{l}^{\tran}Q_0^{-\tran}+
      Q_0^{-1}B_{l} W_2(\vvec)\right)\right)  Q_0^{\tran}).
\]
In Table~\ref{tab1} we summarize the operations that can be performed offline and online. Notice that the $\mathcal{O}(n^3)$ flops needed to compute $W(\vvec)$ come from the final matrix-matrix multiplications by $Q_0$ and $Q_0^{\tran}$ in~\eqref{eq:Xsmalllastline}. These operations can be avoided for $f=\trace$ and $Q_0$ orthogonal, thus reducing the cost of computing $W(\vvec)$ to $\mathcal{O}(n^2)$ flops.

\begin{table}[t!]
  \centering
  \begin{tabular}{cc|cc}
    \multicolumn{2}{c}{\textbf{Offline}} & \multicolumn{2}{c}{\textbf{Online}}\\
    Operation & Flops & Operation & Flops\\
    Compute $X_0$ and $P$ & $\mathcal{O}(n^3)$ & Compute $Z(\vvec)$ & $\mathcal{O}(kn^2)$ \\
    Eig   $A_0=Q_0\Lambda_0Q_0^{-1}$ & $\mathcal{O}(n^3)$ & Solve~\eqref{eq:SMW_linearsystem}& $\mathcal{O}(mk^2n^2)$\\
    Assemble $\mathbf{N}^{\tran}\LL_0^{-1}\mathbf{M}$ & $\mathcal{O}(k^2n^2)$& Compute $W(\vvec)$&$\mathcal{O}(n^3)$\\
    TSVD $\mathcal{N}_1^{\tran}\mathcal{M}_2$, $\mathcal{N}_2^{\tran}\mathcal{M}_1$ & $\mathcal{O}(n^3)$ &  & \\
    Compute $\widetilde Z_{i}$, $ Z_{i}$, $i=1,\ldots,k$ & $\mathcal{O}(n^3)$ && \\
  \end{tabular}
  \caption{Offline and online operations involved in the computation of the solution $X(\vvec)$ to~\eqref{Ljap jdba 1} for a moderate $n$. $m$ denotes the number of iterations needed by GCRO-DR to converge. Depending on the properties of $f$, the cost of computing $W(\vvec)$ reduces to $\mathcal{O}(n^2)$ flops.}\label{tab1}
\end{table}

\REV{We conclude this section by recalling the reader that $w(\vvec)$ has been computed by GCRO-DR, an iterative method that has been run up to a certain tolerance on the relative residual norm associated to~\eqref{eq:SMW_linearsystem}. Therefore, we can think of $w(\vvec)$ as being of the form $w(\vvec)=w_*(\vvec)+\ell(\vvec)$, where $w_*(\vvec)$ is the exact solution to~\eqref{eq:SMW_linearsystem} whereas $\ell(\vvec)$ is the error vector coming from the computation of $w(\vvec)$. The magnitude of $\ell(\vvec)$ plays an important role in the overall accuracy that can be attained by $X(\vvec)$ in~\eqref{eq:Xsmalllastline}. In particular, by writing $\ell(\vvec)=[\ell_1(\vvec);\ell_2(\vvec)]$, $\ell_1(\vvec),\ell_2(\vvec)\in\mathbb{C}^{nk}$, and $E_1(\vvec)\in\mathbb{C}^{n\times k}$, $\ell_1(\vvec)=\text{vec}(E_1(\vvec))$, $E_2(\vvec)\in\mathbb{C}^{k\times n}$, $\ell_2(\vvec)=\text{vec}(E_2(\vvec))$,
a direct computation shows that we can write
\[
  X(\vvec)=X_0+Z(\vvec)+W_*(\vvec)+E(\vvec),
  \]
where $E(\vvec)=Q_0 \left(\cL\mathop{\circ}\left(E_1(\vvec)B_{l}^{\tran}Q_0^{-\tran}+
     Q_0^{-1}B_{l} E_2(\vvec)^{\tran}\right)\right)  Q_0^{\tran}$. To bound $\|E(\vvec)\|_F$ it may be more convenient to look at $E(\vvec)$ as the solution to the following Lyappunov equation
     \[
     A_0E(\vvec)+E(\vvec)A_0^{\tran}=Q_0E_1(\vvec)B_{l}^{\tran}+
     B_{l} E_2(\vvec)^{\tran}Q_0^{\tran}.
     \]
     Indeed, for stable $A_0$, by defining $\alpha_0$ as the largest eigenvalue of the symmetric part of $A_0$, namely $\alpha_0=\lambda_{\max}((A_0+A_0^{\tran})/2)$, if $\alpha_0<0$, standard results on the norm of solutions to Lyapunov equations say that
     \[
     \|E(\vvec)\|_F\leq \frac{1}{2|\alpha_0|}\|Q_0E_1(\vvec)B_{l}^{\tran}+
     B_{l} E_2(\vvec)^{\tran}Q_0^{\tran}\|_F;     \]
     see, e.g.,~\cite{Ves2011}.}
\subsection{Large-scale setting: a projection framework}\label{Large-scale setting: a projection framework}
In this section, we address the numerical solution of problems of large dimensions. In particular, due to a too large value of $n$,  we assume that many of the operations described in the previous section, e.g., the computation of the eigenvalue decomposition of $A_0$, are not affordable. On the other hand, we assume \REV{that it is possible, and efficient, to solve linear systems with $A_0$.}
\begin{Ass}\label{ass:3}
  \REV{Solving linear systems with $A_0$ and possibly multiple right-hand sides is possible and efficient}.
\end{Ass}

\REV{In all our numerical results, we computed (and stored) the LU factors of $A_0$.}

Taking inspiration from well-established projection methods for large-scale Lyapunov equations, we design a novel scheme for~\eqref{eq:Lyap_deltaX}. As before, the main goal is to reuse the information we already have at hand, as much as possible, every time~\eqref{eq:Lyap_deltaX} has to be solved for a new $\vvec$.

We employ the extended Krylov subspace method presented in~\cite{Sim07} \REV{and further studied in~\cite{KniSim2011}.} In this section we illustrate how to fully exploit its properties for our purposes.

In the case of a parameter independent Lyapunov equation of the form~\eqref{eq:seedLE} with \(\REV{Q}=PP^{\tran}\),
the extended Krylov subspace
\begin{equation}\label{eq:EK_def}
  \begin{aligned}
    \textbf{EK}_{m}^\square(A_0,P)&=\text{Range}([P,A_0^{-1}P,A_0P,A_0^{-2}P,\ldots,A_0^{m-1}P,A_0^{-m}P])\\
    &=\textbf{K}_{m}^\square(A_0,P)+\textbf{K}_{m}^\square(A_0^{-1},A_0^{-1}P),
  \end{aligned}
\end{equation}
can be employed as approximation space. In~\eqref{eq:EK_def}, $\textbf{K}_{m}^\square(A_0,P)$ denotes the standard, polynomial (block) Krylov subspace generated by $A_0$ and $P$, namely $\textbf{K}_{m}^\square(A_0,P)=\text{Range}([P,A_0P,\ldots,A_0^{m-1}P])$.

If $V_{m}\in\RR^{n\times 2mk}$, $k=\text{rank}(P)$, has orthonormal columns\footnote{\REV{Hereafter, we assume $V_m$ to have full rank. If this is not the case, standard deflation strategies can be adopted; see, e.g.,~\cite[Section 7.1]{Birk2015}~\cite{Gut2007}.}} and it is such that $\text{Range}(V_{m})=\textbf{EK}_{m}^\square(A_0,P)$, the extended Krylov subspace method computes an approximate solution of the form $X_{m}=V_{m}Y_{m}V_{m}^{\tran}$. The square matrix $Y_{m}\in{\RR}^{2mk\times 2mk}$ is usually computed by imposing a Galerkin condition on the residual matrix $R_{m}=A_0X_{m}+X_{m}A_0^{\tran}+PP^{\tran}$, namely $V_{m}^{\tran}R_{m}V_{m}=0$. It is easy to show that such a condition is equivalent to computing $Y_{m}$ by solving a \emph{reduced} Lyapunov equation.
\REV{With $Y_m$ at hand, an accuracy measure, either the residual norm or the backward error, is computed to assess the quality of the current solution. }
 Whenever this value is sufficiently small, we stop the process, otherwise the space is expanded.

A very similar scheme could also be employed for~\eqref{eq:Lyap_deltaX}. However, both the right-hand side and the coefficient matrix in~\eqref{eq:Lyap_deltaX} depend on $\vvec$ so that, at a first glance, a new extended Krylov subspace has to be computed for each $\vvec$. This would be too expensive from a computational point of view.

We show that only one extended Krylov subspace can be constructed and employed for solving all the necessary Lyapunov equations~\eqref{eq:Lyap_deltaX} regardless of the change in $\vvec$. In particular, we suggest the use of $\textbf{EK}_{m}^\square(A_0,P)$ as approximation space, where $P$ is defined as in~\eqref{eq:rhs_newexpression}. This selection is motivated by the following reasoning.

As reported in~\eqref{eq:rhs_newexpression}, the right-hand side of~\eqref{eq:Lyap_deltaX} can be written in factorized form as
\[
  P \left( \mathcal{I}_2\otimes D(\vvec)\right)
  P^{\tran},
  \qquad P=[X_0 B_{r},\,B_{l}]\in\mathbb{R}^{n\times 2k},
  \qquad \mathcal{I}_2=
  \begin{bmatrix}
    0&1\\
    1& 0\\
  \end{bmatrix}.
\]
Therefore, we can use only $P$ as the starting block for the construction of the approximation space for every $\vvec$. See, e.g.,~\cite[Section 3.2]{LanMS15}.

The argument to justify the use of $A_0$ in place of $A(\vvec)$ is a little more involved. We first observe that we can write
\[A(\vvec)P=(A_0-B_{l}D(\vvec)B_{r}^{\tran})P=A_0P-B_{l}(D(\vvec)B_{r}^{\tran}P),\]
and, since $\text{Range}(B_{l})\subseteq\text{Range}(P)$ by definition, we have $\text{Range}(A(\vvec)P)\subseteq\text{Range}(A_0P)+\text{Range}(P)$. Similarly, we can show that
\[\text{Range}({A(\vvec)}^{j}P)\subseteq\text{Range}(A_0^{j}P)+\text{Range}(P),\]
for all $j>0$. This implies that
$\textbf{K}_{m}^\square(A(\vvec),P)\subseteq \textbf{K}_{m}^\square(A_0,P)$. We now show that also the Krylov subspaces for the inverses are nested, i.e. $\textbf{K}_{m}^\square({A(\vvec)}^{-1},{A(\vvec)}^{-1}P)\subseteq \textbf{K}_{m}^\square(A_0^{-1},A_0^{-1}P)$, so that $\textbf{EK}_{m}^\square(A(\vvec),P)\subseteq \textbf{EK}_{m}^\square(A_0,P)$ and the use of the non-parametric space $\textbf{EK}_{m}^\square(A_0,P)$ is thus justified.

Thanks to the SMW formula, we have
\[
  {A(\vvec)}^{-1}P={(A_0-B_{l}D(\vvec)B_{r}^{\tran})}^{-1}P=A_0^{-1}P+A_0^{-1}B_{l}{({D(\vvec)}^{-1}-B_{r}^{\tran}A_0^{-1}B_{l})}^{-1}B_{r}^{\tran}A_0^{-1}P,
\]
and once again, since $\text{Range}(A_0^{-1}B_{l})\subseteq\text{Range}(A_0^{-1}P)$, it holds $\text{Range}({A(\vvec)}^{-1}P)\subseteq\text{Range}(A_0^{-1}P)$.
Similarly, $\text{Range}({A(\vvec)}^{-j}P)\subseteq\text{Range}(A_0^{-j}P)$ for all $j>0$. Therefore, we can use $\textbf{EK}_{m}^\square(A_0,P)$ for the computation of $X(\vvec)$.

If $V_{m}\in\RR^{n\times 4mk}$ has orthonormal columns, and it is such that $\text{Range}(V_{m})=\textbf{EK}_{m}^\square(A_0,P)$, we seek an approximate solution $X_{\delta,m}(\vvec)=V_{m}Y_{m}(\vvec)V_{m}^{\tran}$ to~\eqref{eq:Lyap_deltaX}. As in the nonparametric case, the matrix $Y_{m}(\vvec)\in\RR^{4mk\times 4mk}$ is computed by imposing a Galerkin condition on the residual, namely it is the solution of the projected equation
\begin{equation}\label{eq:param_projected}
  (T_{m} - B_{l,m}D(\vvec)B_{r,m}^{\tran}) Y(\vvec) + Y(\vvec)
  {(T_{m}-B_{l,m}D(\vvec)B_{r,m}^{\tran})}^{\tran} = E_1G(\mathcal{I}_2\otimes D(\vvec))G^{\tran}E_1^{\tran},
\end{equation}
where $T_{m}=V_{m}^{\tran}A_0V_{m}$ can be cheaply computed as suggested in~\cite{Sim07}, $B_{l,m}=V_{m}^{\tran}B_{l}$ and $B_{r,m}=V_{m}^{\tran}B_{r}$ can be updated on the fly by performing $4k$ inner products, whereas $G\in\RR^{4k\times 2k}$ is such that $P=V_1G$. Notice that $B_l$ can be exactly represented in $\textbf{EK}_{m}^\square(A_0,P)$, namely $B_l=V_{m}B_{l,m}$, as it is part of the initial block $P=[X_0B_r,B_l]$. This does not hold for $B_r$.

Since the Galerkin method is structure preserving,
the small-dimensional equation~\eqref{eq:param_projected} can be solved by the procedure presented in Section~\ref{Small-scale setting: recycling Krylov methods}. \REV{Notice that the well-posedness of equation~\eqref{eq:param_projected} is difficult to guarantee a priori without any further assumptions on $B_l$, $B_r$, and $D(\vvec)$.  Indeed, even in case of an $A_0$, and thus $T_m$, definite, the signature of $T_{m} - B_{l,m}D(\vvec)B_{r,m}^{\tran}$ cannot be easily predicted. However, in Section~\ref{subsec:application-damped_vibrational_systems}, we will show that the structural properties, inherent to some of the application settings we are interested in, ensure to work with a matrix $A(\vvec)$ that is stable for every $\vvec$. Even though this does not guarantee equation~\eqref{eq:param_projected} to be well-defined for any $m$, it has been noticed in~\cite[Section 5.2.1]{Simoncini2016} that it is very unlikely to get singular projected equations in this scenario.}

Once $Y_m(\vvec)$ has been computed, we need to decide whether the corresponding numerical solution $X_{\delta,m}(\vvec)=V_{m}Y_{m}(\vvec)V_{m}^{\tran}$ is sufficiently accurate.
\REV{To this end, we employ the backward error which can be computed as shown in the following proposition for the Lyapunov equation~\eqref{eq:Lyap_deltaX}}.
\begin{Prop}\label{Prop:backward_error}
  The backward error $\Delta_{m,\vvec}=\frac{\|R_{m}(\vvec)\|_F}{2\|A(\vvec)\|_F\|X_{\delta,m}(\vvec)\|_F+\|P(\mathcal{I}_2\otimes D(\vvec))P^{\tran}\|_F}$ provided by the approximate solution $X_{\delta,m}(\vvec)=V_{m}Y_{m}(\vvec)V_{m}^{\tran}$ to~\eqref{eq:Lyap_deltaX} is such that
  \begin{equation}\label{eq:bacward_error_comput}
    \Delta_{m,\vvec}=
    \frac{\sqrt{2}\|E_{m}^{\tran}\underline{T}_{m}Y_{m}(\vvec)\|_F}%
    {2\sqrt{{\|A_0\|}_F^2 + \vvec^{\tran}((B_{r}^{\tran}B_{l})\mathop{\circ}(B_{r}^{\tran}B_{l}))\vvec-2a^{\tran}\vvec}
      \cdot {\|Y_{m}(\vvec)\|}_F + {\|G(\mathcal{I}_2\otimes D(\vvec))G^{\tran}\|}_F},
  \end{equation}
  where $a=\mathop{\mathrm{diag}}(B_{r}^{\tran}A_0B_{l})$.
\end{Prop}
\begin{proof}
  The expression for the residual matrix
  \[ R_{m}=A(\vvec)X_{m}(\vvec) + X_{m}(\vvec){A(\vvec)}^{\tran} - P (\mathcal{I}_2\otimes D(\vvec))P^{\tran},\] directly comes from the Arnoldi relation
  \[
    A_{0}V_{m}=V_{m}T_{m} + \mathcal{V}_{m+1}E_{m}^{\tran} \underline{T}_{m},
  \]
  where $\mathcal{V}_{m+1} \in \RR^{n\times 4k}$ is the $(m+1)$-th basis block, namely $V_{m+1}=[V_{m},\mathcal{V}_{m+1}]$. We can write
  \begin{align*}
    R_{m}(\vvec)=&A(\vvec)X_{\delta,m}(\vvec)+X_{\delta,m}(\vvec){A(\vvec)}^{\tran}-P (\mathcal{I}_2\otimes D(\vvec))P^{\tran}\\
    =&A(\vvec)V_{m}Y_{m}(\vvec)V_{m}^{\tran}+V_{m}Y_{m}(\vvec)V_{m}^{\tran}{A(\vvec)}^{\tran}-P
       (\mathcal{I}_2\otimes D(\vvec))P^{\tran}\\
    =&(A_0-B_{l} D(\vvec)B_{r}
       ^{\tran})V_{m}Y_{m}(\vvec)V_{m}^{\tran}+V_{m}Y_{m}(\vvec)V_{m}^{\tran}(A_{0}-B_{l}
       D(\vvec)B_{r}^{\tran})-P (\mathcal{I}_2\otimes
       D(\vvec))P^{\tran}.
  \end{align*}
  Recalling that $B_{l}=V_{m}B_{l,m}$, $P=V_{m}E_{1}G$, and plugging the Arnoldi relation above, we get
  \begin{align*}
    R_{m}(\vvec)=&V_{m}((T_{m}-B_{l,m}D(\vvec)B_{r,m}^{\tran})Y(\vvec)+Y(\vvec){(T_{m}-B_{l,m}D(\vvec)B_{r,m}^{\tran})}^{\tran}\\
                 &-E_1G(\mathcal{I}_2\otimes
                   D(\vvec))G^{\tran}E_1^{\tran})V_{m}^{\tran} +
                   \mathcal{V}_{m+1}E_{m}^{\tran}\underline{T}_{m}Y_{m}(\vvec)V_{m}^{\tran}
                   + V_{m}Y_{m}(\vvec)
                   \underline{T}_{m}^{\tran}E_{m}\mathcal{V}_{m+1}^{\tran}\\
    =&\mathcal{V}_{m+1}E_{m}^{\tran}\underline{T}_{m}Y_{m}(\vvec)V_{m}^{\tran}+V_{m}Y_{m}(\vvec)
       \underline{T}_{m}^{\tran}E_{m}\mathcal{V}_{m+1}^{\tran}.
  \end{align*}
  Therefore, thanks to the orthogonality of $V_{m}$ and $\mathcal{V}_{m+1}$, we have
  \[{\|R_{m}(\vvec)\|}_F^2=2{\|E_{m}^{\tran}\underline{T}_{m}Y_{m}(\vvec)\|}_F^2.\]
  Similarly, since $P=V_1G$,
  \[{\|P(\mathcal{I}_2\otimes D(\vvec))P^{\tran}\|}_F = {\|G(\mathcal{I}_2\otimes D(\vvec))G^{\tran}\|}_F.\]
  We now focus on $\|A(\vvec)\|_F$. By exploiting the cyclic property of the trace operator, and recalling that $D(\vvec)$ is diagonal, we have
  \begin{align*}
    {\|A(\vvec)\|}^2_F
    = &{\|A_0-B_{l} D(\vvec)B_{r}^{\tran}\|}_F^2 =
        {\|A_0\|_F^2+\|B_{l} D(\vvec)B_{r}^{\tran}\|}_F^2-2\trace{(B_{l} D(\vvec)B_{r}^{\tran}A_0)}\\
    =&{\|A_0\|}_F^2+\trace{(B_{l} D(\vvec)B_{r}^{\tran}B_{l} D(\vvec)B_{r}^{\tran})}-2\trace{( D(\vvec)B_{r}^{\tran}A_0B_{l})}\\
    =&{\|A_0\|}_F^2+\trace{( D(\vvec)(B_{r}^{\tran}B_{l}) D(\vvec)(B_{r}^{\tran}B_{l}))}-2a^{\tran}\vvec\\
    =&{\|A_0\|}_F^2+\vvec^{\tran}((B_{r}^{\tran}B_{l})\mathop{\circ}(B_{r}^{\tran}B_{l}))\vvec-2a^{\tran}\vvec,
  \end{align*}
  where
  $a=\mathop{\mathrm{diag}}(B_{r}^{\tran}A_0B_{l})$.

\end{proof}
Proposition~\ref{Prop:backward_error} shows that even though the computation of $\Delta_{m,\vvec}$ depends on the current parameter $\vvec$, its evaluation can be carried out at low cost.
In fact, quantities involving the full problem dimension \(n\), like $\|A_0\|_F$, $(B_{r}^{\tran}B_{l})\mathop{\circ}(B_{r}^{\tran}B_{l})$, and $a$, can be computed in the offline stage. On the other hand, the computational effort for the remaining parameter-dependent terms is only quadratic in the dimension \(k\) of the parameter space. Thus, as long as \(k^{2}\) stays well below \(n\), this procedure is cheap.
If the accuracy provided by $Y_{m}(\vvec)$ is not adequate, we expand the space.

With the corresponding $Y_{m}(\vvec)$ at hand, we can compute $f(X_{\delta,m}(\vvec))=f(V_{m}Y_{m}(\vvec)V_{m}^{\tran})$. Notice that, depending on the nature of $f$, the structure of $X_{\delta,m}(\vvec)$ can be further exploited. For instance, we can cheaply evaluate $\trace{(X_{\delta,m}(\vvec))}$, as
\[
  \trace{(X_{\delta,m}(\vvec))} = \trace{(V_{m}Y_{m}(\vvec)V_{m}^{\tran})}=\trace{(Y_{m}(\vvec))},
\]
thanks again to the cyclic property of the trace and the orthogonality of $V_{m}$. Thus, the basis $V_{m}$ is not necessary to compute \(\trace{(X_{\delta,m}(\vvec))}\).

The overall procedure is summarized in Algorithm~\ref{EKSM_opt}. Note that Algorithm~\ref{EKSM_opt} can be easily modified to be used in optimization procedures, having instead of the set of parameter vectors $\mathsf{V}$ a starting vector $\vvec_0$ and using Algorithm~\ref{EKSM_opt} to calculate
$f(V_{m}Y_{m}(\vvec_\ell)V_{m}^{\tran})$ for each new $\vvec_\ell$ in the iterative optimization scheme while keeping, and possibly expanding, the computed subspace.
\begin{algorithm}[t]
  \DontPrintSemicolon
  \SetKwInOut{Input}{input}\SetKwInOut{Output}{output}
  \Input{$A_0\in\RR^{n\times n}$, $\mathsf{V}$: the set of parameter vectors,  $(B_{r}^{\tran}B_{l})\mathop{\circ}(B_{r}^{\tran}B_{l})\in\RR^{k\times k}$, $a=\mathop{\mathrm{diag}}(B_{r}^{\tran}A_0B_{l})\in\RR^{k}$, $P\in\RR^{n\times 2k}$, $\|A_0\|_F$
    \newline max.\ iteration count $m_{\max}$, backward error bound $\epsilon>0$.}
  \Output{$f{(X_{\delta,m}(\vvec))}$ for all $\vvec \in \mathsf{V}$.}
  \BlankLine
  Compute a skinny QR factorization of $[P,A_0^{-1}P]=V_1[\gamma,\theta]$\;
  Set $m=1$, $\mathsf{V}_0 = \mathsf{V} $\;
  \While{ $m< m_{\max}$}{
    Compute next basis block $\mathcal{V}_{m+1}$, set $V_{m+1}=[V_{m},\mathcal{V}_{m+1}]$, and update $T_{m}=V_{m}^{\tran}A_0V_{m}$ as in~\cite{Sim07}\;
    Update $B_{l,m}=V_{m}^{\tran}B_{l}$, $B_{r,m}=V_{m}^{\tran}B_{r}$\;
    Compute the offline quantities in Table~\ref{tab1} for solving equation~\eqref{eq:param_projected}\;
    \For{all $\vvec \in \mathsf{V}_0$}{
      Perform the online steps in Table~\ref{tab1} and compute  $Y_{m}(\vvec)$\;
      Compute $\Delta_{m,\vvec}$ as in~\eqref{eq:bacward_error_comput}\;
      \If{$\Delta_{m,\vvec}\leq\epsilon$}{
        Compute $f(V_{m}Y_{m}(\vvec)V_{m}^{\tran})$\;
        $\mathsf{V}_0 = \mathsf{V} \setminus \left\{ \vvec \right\}$
      }
      \Else{\textbf{break} and go to line~\ref{alg:last_line}
      }
    }
    Set $m=m+1$\label{alg:last_line}\;
  }
  \caption{Extended Krylov subspace method for calculating trace of a structured, parameter dependent Lyapunov equation.}%
  \label{EKSM_opt}
\end{algorithm}

We recall that the solution of the linear systems with $A_0$ needed in the construction of the basis of $\mathbf{EK}_{m}^\square(A_{0},P)$ can be efficiently carried out thanks to Assumption~\ref{ass:3}.

\REV{As outlined above, the construction of the extended Krylov method requires adding $4k$ vectors per iteration to the current basis even though the initial block $P$ has rank $2k$. This could lead to the allocation of a large basis if many iterations $m$ need to be performed. To avoid the use of an excessive amount of memory, one could combine our methodology with the compress-and-restart paradigm presented in~\cite{KreLMetal21} for the polynomial Krylov subspace method.}

\section{Applications}\label{Applications} 
In this section we illustrate two important problem settings where~\eqref{Ljap jdba 1} needs to be solved several times, for many parameters $\vvec$. The large number of equations in these scenarios makes our novel solvers very appealing, especially if compared to state-of-the-art procedures, which are not able to fully capitalize on the structure of the problem; see also Section~\ref{Numerical examples}.

\subsection{Damped vibrational systems}
\label{subsec:application-damped_vibrational_systems}

We consider linear vibrational systems described by
\begin{equation}\label{MDK}
  \begin{gathered}
    M\ddot{x}+C(\vvec)\dot{x}+Kx=0,\\
    x(0) = x_0, \quad \dot{x}(0) = \dot{x}_0,
  \end{gathered}
\end{equation}
where $M$ denotes the mass matrix, $C(\vvec)$ denotes the parameter dependent damping matrix, $K$ denotes the stiffness matrix, and $x_0$ and $\dot{x}_0$ are initial data. We assume that $M$ and $K$ are real, symmetric positive definite matrices of order $m$. The damping matrix is defined as $C(\vvec)=C_{\rm int}+C_{\rm ext} (\vvec)$.

We assume that
\begin{equation}\label{Cext}
  C_{\rm ext}(\vvec)=B D(\vvec)B ^{\tran},
\end{equation}
where the matrix $B\in \RR^{m\times k}$ describes the dampers' geometry and the matrix $D(\vvec)=\diag{(v_1,v_2,\ldots,v_k)}$ contains the damping viscosities $v_{i} > 0$, for $i=1\ldots,k$. The viscosities will be  encoded in the parameter vector
$\vvec=\begin{bmatrix}v_1, v_2, \ldots,v_k\end{bmatrix}^{\tran}\in  \RR^{k}_{> 0}$.

The internal damping $C_{\rm int}$ can be modeled in different ways.  It is usually modeled as a Rayleigh damping matrix, which means that
\begin{equation}\label{Cprop}
  C_{\rm int}=\alpha M + \beta K,\quad \mbox{with}\quad  \alpha, \beta\geq 0, \; \alpha^2 + \beta^2 >0,
\end{equation}
or as a small multiple of the critical damping, that is,
\begin{equation}\label{Ccrit}
  C_{\rm int} =\alpha M^{1/2}\sqrt{M^{-1/2}KM^{-1/2}}M^{1/2}, \quad \text{where}\quad \alpha> 0.
\end{equation}
See, e.g.,~\cite{KTTmodaldamping12, Ves2011, BB80} for further details.

Vibrations arise in a wide range of systems, such as mechanical, electrical or civil engineering structures. Vibrations are a mostly unwanted phenomenon in vibrational structures, since they can lead to many undesired effects such as creation of noise, oscillatory loads or waste of energy that may produce damaging effects on the considered structures. Therefore, in order to attenuate or minimize unwanted vibrations, an important problem is to determine the damping matrix in such a way that the vibrations of the system are as small as possible. This is usually achieved through optimization of the external damping matrix \(C_{\rm ext} (\vvec)\). Within this framework, we will focus on the optimization of the damping parameter vector $\vvec$ defining \(C_{\rm ext} (\vvec)\) as in~\eqref{Cext}.  While damping optimization is a widely studied topic, there are still many challenging tasks that require efficient approaches. There is a vast literature in this field of research. For further details, we list only a few references that address the minimization of dangerous vibrations with different applications~\cite{beards1996structural,  genta2009vibration, Ves2011, ITakewaki2009,  MullerSchiehlen85, du2016modeling, ZuoNay05}.

The problem of vibrations minimization requires a proper optimization criterion. A whole class of criteria are based on eigenvalues; see, e.g.~\cite{FreitasLancaster99, MoroEgana16, WJSH2018, GMQSW2016, Cox:98, TissMeer01}. Another important criterion is based on the total average energy of the system.  This has been intensely considered in the last two decades; see, e.g.,~\cite{Ves2011, BRAB98, TRUHVES09, TrTomPuv17, Cox:04, Nakic13}.
Since our approach is also based on the total average energy, in the following we provide more details and set the stage for the application of our framework.

Several approaches trying to fully exploit the structure and accelerate the optimization process have been proposed in the literature about damping systems. In more details, in~\cite{VES89, VES90, TTVcasestudy11} the authors considered approaches that allowed derivation of explicit formulas for the total average energy, but they are adequate only for certain case studies. Furthermore, in~\cite{BennerTomljTruh10, morBenTT13}  the authors employed dimension reduction techniques in order to obtain efficient approaches for the calculation of the total average energy. However, these approaches require specific system configurations and they cannot be applied efficiently for general system matrices.

If we write~\eqref{MDK} as a first-order ODE in phase space $\dot{y} = A(\vvec) y$, the solution is given by $y (t) = \mathrm{e}^{A(\vvec)t}y_0$, where $y_0$ denotes the vector of initial data. The average total energy of the system is given by
\[
  \int_{\lVert y_0 \rVert  = 1} \int_0^{\infty} {y(t)}^{\tran} y(t) \,\mathrm{d}t \,\mathrm{d}\sigma
  = \int_{\lVert y_0 \rVert  = 1} \int_0^{\infty} y_0^{\tran} \mathrm{e}^{{A(\vvec)}^{\tran}t}\mathrm{e}^{A(\vvec)t}y_0 \,\mathrm{d}t \,\mathrm{d}\sigma,
\]
where $\sigma$ is a given non-negative measure on the unit sphere. From~\cite[Proposition 21.1]{Ves2011} it follows that one can calculate the average total energy as $\trace (X(\vvec))$, where $X(\vvec)$ solves the Lyapunov equation
\begin{equation}\label{average total energy}
  A(\vvec) X(\vvec) + X(\vvec){A(\vvec)}^{\tran} = -\REV{Q}.
\end{equation}
Here $\REV{Q}$ is the unique positive semidefinite matrix determined by $\sigma$.
We will show that $A(\vvec)$ has the form~\eqref{eq:formA} and that Assumptions~\ref{ass:2} is satisfied for both types of internal damping mentioned above and for typical choices of $\sigma$.

Using the linearization $y_1 = x$, $y_2 = \dot{x}$, the differential equation (\ref{MDK}) in phase
space can be written as
\begin{equation}\label{MatrixAA}
  \dot{y}=A(\vvec) y,\qquad \mbox{where}\qquad  A(\vvec)=\left[%
    \begin{array}{cc}
      0 & I \\
      -M^{-1}K & - M^{-1}C(\vvec) \\
    \end{array}%
  \right].
\end{equation}
In the case of internal Rayleigh damping~\eqref{Cprop}, it is easy to see that $A(\vvec)$ has the form~\eqref{eq:formA} with
\[
  A_0 = \begin{bmatrix}
    0 & I \\ - M^{-1} K & - \alpha I - \beta M^{-1}K
  \end{bmatrix}, \quad
  B_l = \begin{bmatrix}
    0 \\ M^{-1}B
  \end{bmatrix},  \quad
  B_r = \begin{bmatrix}
    0 \\ B
  \end{bmatrix}.
\]
It can be shown that $A(\vvec)$ is a Hurwitz matrix (the eigenvalues
of $A$ are in the open left half of the complex plane) if $\alpha^2 + \beta^2 > 0$. Moreover, by using the appropriate permutation matrix, $A_0$ can be written as a block diagonal matrix with $2\times 2$ blocks on its diagonal so that its eigenvalues can be directly calculated.
Assume $Z = \mathop{\mathrm{diag}}(K^{-1},M^{-1})$, which corresponds to the choice of the surface measure  $\sigma$ generated by the energy norm $\lVert y \rVert = \frac{1}{\sqrt{2}}\sqrt{ y_1^{\tran} K y_1 +  y_2^{\tran} M y_2} $ on $\RR^{2m}$.
Then, one can calculate the corresponding $X_0$ as
\begin{equation}
  X_0 =  \begin{bmatrix}
    {(\alpha K + \beta K M^{-1}K)}^{-1} + \alpha K^{-1}M^{-1}K^{-1} + \beta  K^{-1} & -\frac{1}{2}K^{-1} \\
    -\frac{1}{2}K^{-1} & {(\alpha K + \beta K M^{-1}K)}^{-1}
  \end{bmatrix},
\end{equation}
hence Assumption~\ref{ass:2} is satisfied.
It can be shown that the eigenvalues of $A_0$ are given by
\[
  \frac{1}{2} \left(- \alpha - \beta \lambda_{i} \pm \sqrt{{(\alpha + \beta \lambda_{i})}^2 - 4 \lambda_{i}}\right),
\]
where $\lambda_{i} > 0$, $i=1,\ldots,m$, are the eigenvalues of the matrix pair $(K,M)$ and that the corresponding eigenvectors can be constructed from the eigenvectors of the pair $(K,M)$. Hence, Assumption~\ref{ass:3} is satisfied.

In the case of internal damping of the form~\eqref{Ccrit}, a different kind of linearization is more convenient.
From the assumptions on $M$ and $K$ it follows that there exists a matrix
$\Phi$ that simultaneously diagonalizes
$M$ and $K$, i.e.,
\begin{equation}\label{MatrixPhi}
  \Phi^{\tran} K \Phi = \Omega^2
  \quad \text{and}
  \quad \Phi^{\tran} M \Phi = I,
\end{equation}
where $\Omega = \diag (\omega_1 , \ldots, \omega_{m}  )$ contains the square roots of the eigenvalues of $(K,M)$, which are eigenfrequencies of the corresponding undamped ($C(\vvec)=0$) vibrational system. We assume that they are ordered in ascending order,  $0 < \omega_1 \leq\omega_2\leq\cdots\leq \omega_{m} $.
It holds that the matrix $\Phi$ diagonalizes $C_{\rm int}$ as well, that is,
\begin{equation*}
  \Phi^{\tran} C_{\rm int} \Phi = \alpha \Omega,
\end{equation*}
for more details see, e.g.,~\cite{Ves2011}.

Using the linearization $y_1 = L_K^{\tran}x$, $y_2 = L_M^{\tran}\dot{x}$, where $L_K$, $L_M$ are the Cholesky  factors of $K$ and $M$, respectively, the matrix $A(\vvec)$ can be written as
\begin{equation}\label{MatrixA}
  A(\vvec)=\left[%
    \begin{array}{cc}
      0 & \Omega \\
      -\Omega & -\Phi^{\tran} C(\vvec) \Phi \\
    \end{array}%
  \right],
\end{equation}

where $\Omega$ and $\Phi$ are given by~\eqref{MatrixPhi}.
\REV{It is important to notice that the matrix $A(\vvec)$ from (\ref{MatrixA}) is Hurwitz if $\alpha >0$; see, e.g.,~\cite{Ves2011}.}

Since in this case the matrix $A_0$ is given by
\[
  A_0 = \begin{bmatrix}
    0 & \Omega \\ - \Omega & - \alpha \Omega
  \end{bmatrix},
\]
 all its eigenvalues of $A_0$ are non-real for $\alpha < 2$.

Let $n = 2m$. Typical choices for the matrix $\REV{Q}$ in this case are $\REV{Q}=\frac{1}{n}I$, which corresponds to the case when $\sigma$ is generated by the Lebesgue measure on $\RR^{n\times n}$, and, for $s< m$,
\begin{equation}\label{matrix Z}
  \REV{Q}=\frac{1}{2s}
  \begin{bmatrix}
    I_s & 0 & 0 & 0 \\
    0 & 0 & 0 & 0\\
    0 & 0 & I_s & 0 \\
    0 & 0 & 0 & 0
  \end{bmatrix}, \end{equation}
which corresponds to the case when $\sigma$ is generated by the Lebesgue measure on the subspace spanned by the vectors
${[x_{i}, 0]}^{\tran}$ and ${[0, x_{i}]}^{\tran}$, $i = 1,\ldots, s$, where $x_{i}$ are the eigenvectors of the first $s$ $\omega_{i}$'s and on the rest of $\RR^n$ it corresponds to the Dirac measure concentrated at zero.

For both choices of the matrix $\REV{Q}$ mentioned above, Assumption~\ref{ass:2} is satisfied. Indeed,
if $\REV{Q}=\frac{1}{n}I$, we have
\[
  X_0=\frac{1}{n}
  \begin{bmatrix}
    \frac{3 \alpha}{2}\Omega^{-1} & -\frac{1}{2}\Omega^{-1}\\
    -\frac{1}{2}\Omega^{-1} & \frac{1}{\alpha}\Omega^{-1}
  \end{bmatrix},
\]
and if $\REV{Q}$ is given by~\eqref{matrix Z}, then a direct calculation shows that
\[
  X_0 = \frac{1}{2s}
  \begin{bmatrix}
    \frac{3 \alpha}{2}\Omega_s^{-1} & 0 & -\frac{1}{2}\Omega_s^{-1} & 0\\
    0 & 0 & 0 & 0 \\
    -\frac{1}{2}\Omega_s^{-1} & 0 & \frac{1}{\alpha}\Omega_s^{-1} & 0 \\
    0 & 0 & 0 & 0
  \end{bmatrix},
\]
where $\Omega_s = \mathop{\mathrm{diag}}(\omega_1,\ldots,\omega_s)$.

\subsection{Multi-agent systems}
\label{subsec:multi-agent systems}

The second application concerns the analysis of parameter variation in output synchronization for heterogeneous multi-agent systems.  In a nutshell, multi-agent systems are a class of dynamical systems on networks, consisting of a number of dynamical systems (\emph{agents}) connected in a network, where the agents can exchange information with the neighboring agents through a given interaction \emph{protocol}, with the communication topology being specified through a (combinatorial) graph. Agent dynamics can be rather complex; see, e.g.,~\cite{Nojavanzadeh21}. Examples of systems that can be modelled as multi-agent systems are, e.g.,\ wireless sensor networks, power grids, and social networks. Common protocols are those for which the systems achieve consensus or output synchronization, and those for which the systems achieve desired formation or flocking state. The design and analyses of multi-agent systems have been a widely investigated field in recent years and for more details see, e.g.,~\cite{MeshabiEgerstedt2010,Yu2017,wang2021robust} and references therein.

The problem of output synchronization is how to design / control a system in such a way that the outputs of the agents converge to the same state. The general setting of output synchronization problems for heterogeneous multi-agent systems was considered, e.g., in~\cite{Kim11,Davodi14,Jiao21,Nojavanzadeh21}.

An important aspect in heterogeneous multi-agent systems is studying the impact that one agent has, or a subset of agents has, on the whole system. This will, of course, depend on the dynamics of these agents and their location in the communication graph. The goal of this subsection is to analyze this impact with high computational efficiency in the case of the output synchronization protocol. We will analyze how one or more agents influence the entire system by using the $H_2$-norm (see, e.g.,~\cite{zhou1998essentials}) of the corresponding multi-agent system as a performance measure. By using this information, the system designer can modify / control the heterogeneous multi-agent system in order to achieve a target behavior.

To understand why equations of the form~\eqref{Ljap jdba 1} arise in this setting, we first need to briefly introduce the bigger picture. Motivated by work from~\cite{Palunko21,RenBeard08,NakTolTomljPal22}, we consider $m$ agents with their dynamics described by
\begin{equation}
  \label{eq:MASsystem}
  \begin{aligned}
    \dot \xi_{i} &= A^{(i)} \xi_{i} + B^{(i)}u_{i} + \omega_{i},\\
    \zeta_{i}&=C^{(i)}\xi_{i}.
  \end{aligned}
\end{equation}
Here the function $\xi_{i}$ represents the state of the agent $i$, the function  $u_{i}$ represents the input of the agent $i$, the function $\zeta_{i}$ represents the output of the agent $i$, and the function $\omega_{i}$ represents the exogenous disturbance of the agent $i$, $i \in \{1, 2, \ldots, m \}$.
The matrices $A^{(i)}$, $B^{(i)}$ and $C^{(i)}$ are called state, input, and output matrices, of the agent  $i$, respectively. We assume that all matrices $A^{(i)}$ have the same size, and that the same holds for the matrices $B^{(i)}$ and $C^{(i)}$.
Let $G$ be the corresponding communication graph, which models the connections of the agents from~\eqref{eq:MASsystem}, with nodes $\left\{ 1,\ldots,m \right\}$. This means that $(i,j)$ is an edge in the graph $G$ if the agents $i$ and $j$ exchange information through the common protocol.
A protocol we are going to study is the following
\begin{equation}
  \label{eq:protocol}
  u_{i}(t) = - K^{(i)}  \sum_{j \in \mathfrak{N}_{i}}   (\zeta_{i}(t)-\zeta_{j}(t)), \quad i=1,\ldots, m.
\end{equation}
Here $\mathfrak{N}_{i}$ denotes the set of neighboring agents of the agent $i$ (described by $G$) and $K^{(i)}$ is a so-called gain matrix of appropriate dimension~\cite{RenBeard08}.

We define the stack vector $\xi=(\xi_1,\ldots,\xi_m)$. Let $n$ be the size of the matrices $A^{(i)}$, hence $nm$ is the size of the vector $\xi$.
As we want to treat scenarios in which the disturbances $\omega_i$ need not all be independent, let us suppose that among the exogenous disturbances $\{\omega_1,\ldots,\omega_m\}$ there are $k\in\{1,\ldots,m\}$ different ones $\{\omega_{i_1}, \ldots, \omega_{i_k}\}$ (e.g., the same ocean waves or wind gusts concurrently disturb several agents). Then $\begin{bmatrix}
  \omega_1 & \cdots & \omega_m
\end{bmatrix}^\top = H \begin{bmatrix}
  \omega_{i_1} & \cdots & \omega_{i_k}
\end{bmatrix}^\top$, where the matrix $H\in \mathbb{R}^{m\times k}$ is given by
\[
  H_{jl} = \begin{cases}
    1, & \omega_j = \omega_{i_l} \\
    0, & \text{otherwise}
  \end{cases},
  \quad j=1,\ldots, m,\;\; l=1,\ldots,k.
\]
With $\mathbf{E}=H\otimes I_n$ we denote the corresponding disturbance matrix. Let $\omega = (\omega_{i_1},\ldots,\omega_{i_k})$. With $\mathbf{C}$ we denote the overall output matrix $\mathbf{C}=\mathop{\mathrm{diag}}(C^{(1)},\ldots,C^{(m)})$.

The closed-loop dynamic equation of the system described by~\eqref{eq:MASsystem} and~\eqref{eq:protocol} can be represented as
\begin{equation}
  \label{eq:MASsystemCompactForm}
  \begin{aligned}
    \dot \xi  (t)& = \mathbf{A}\xi(t)+\mathbf{E}\omega(t), \\
    \zeta (t)& = \mathbf{C}\xi (t),
  \end{aligned}
\end{equation}
where
the system matrix $\mathbf{A}\in\mathbb{R}^{nm\times nm}$ in~\eqref{eq:MASsystemCompactForm} can be seen as a block matrix whose $(i,j)$-th block $\mathbf{A}_{ij}$ is given by
\begin{equation*}
  \mathbf{A}_{ij} = \delta_{i,j}A^{(i)} -\ell_{ij}B^{(i)}K^{(i)}C^{(j)}, \quad i,j\in \left\{ 1,\ldots,m \right\},
\end{equation*}
where  $\ell_{ij}$, $i,j=1,2,\ldots,m$, are the entries of the Laplacian matrix $L$ of the graph $G$,
see, e.g.,~\cite{Palunko21}. Note, however, that the model in~\cite{Palunko21} has delays, while our model has no delays.

By calculating the $H_2$-norm of the system~\eqref{eq:MASsystemCompactForm}, we are calculating the norm of the mapping $\omega\mapsto \zeta$, thus measuring how the disturbance $\omega$ is influencing the multi-agent system.
This boils down to the calculation of $\trace (\mathbf{E} \mathbf{E}^{\tran} X)$, where $X$ solves the Lyapunov equation $\mathbf{A}^{\tran} X + X\mathbf{A} = - \mathbf{C}^{\tran} \mathbf{C}$. If we were interested in measuring the influence of the disturbance to a part of the system, i.e., to a subset of agents, instead of the matrix $\mathbf{C}$ in~\eqref{eq:MASsystemCompactForm}, we would use the matrix $\mathbf{P}\mathbf{C}$, where $\mathbf{P}$ has the form $\mathbf{P} = P\otimes I$, with $P$ being an orthogonal projector to the subspace spanned by the canonical vectors corresponding to the subset of agents we are interested in.

We want to study the impact that a variation of dynamics in one or a subset of agents has on the $H_{2}$-norm of the resulting system.
To this end, we need to compute $\trace (\mathbf{E} \mathbf{E}^{\tran}
X(\vvec))$ where now $ X(\vvec)$ denotes the solution to a Lyapunov equation of
the form~\eqref{Ljap jdba 1} with $\REV{Q}=\mathbf{C}^{\tran} \mathbf{C}$. In this
framework, the low-rank modification $B_{l}D(\vvec){B_{r}}^{\tran}$ in the coefficient matrix~\eqref{eq:formA} is meant to take into account the variation in the dynamics of the agents of interest. In particular, $B_{l}$ and $B_{r}$ will encode the location of the agents to be modified whereas $D(\vvec)$ amounts to the parameter variation to be applied to those agents.

In the numerical examples in Section~\ref{Numerical examples} we choose symmetric matrices $A^{(i)}$ and $K^{(i)}$, and we will also take $B^{(i)} = (C^{(i)})^{\tran}$, $i=1,\ldots,m$. Hence, the matrix $\mathbf{A}$ will be symmetric and therefore Assumptions~\ref{ass:2} and~\ref{ass:3} will be satisfied.


\section{Numerical examples}\label{Numerical examples}

We now illustrate the advantages of the proposed methods for the two applications described in the previous section. Firstly, we consider an example including the analysis of the influence of damping parameters in a mechanical system. Secondly, we study the impact that variation of agents' dynamics has on a couple of different multi-agent systems. To this end, we employ our novel schemes to efficiently compute the $H_2$-norm of the underlying linear time-invariant system~\eqref{eq:MASsystemCompactForm}.

\begin{example}\label{ex1}
  We consider a mechanical system of $2d+1$ masses consisting of two main rows of $d$ masses connected with springs; see, e.g.,~\cite{morBenTT13,BeaGT20}. The springs in the first row of masses have stiffness $k_1$ and those in the second row have stiffness $k_2$. The first masses, on the left edge,  (i.e., masses $m_1$ and $m_{d+1}$) are connected to  a fixed boundary, while, on the other side of the rows, the last masses ($m_d$ and $m_{2d}$)  are connected to the mass $m_{2d+1}$ which, via a spring with stiffness $k_3$, is connected to a fixed boundary.

  \begin{figure}[t]
    \begin{center}
      \includegraphics[width=10cm]{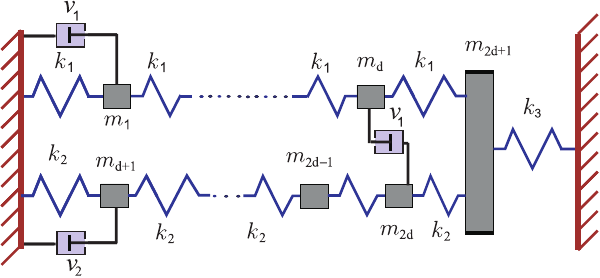}
    \end{center}%
    \caption{Example~\ref{ex1}, illustration of the mechanical system.}%
    \label{Fig2dp1Oscillator}
  \end{figure}

  The model is given by~\eqref{MDK} where the mass matrix is $M  =\diag{(m_1,m_2,\ldots,m_{2d+1})} $  and the  stiffness matrix $K$ is given by
  \begin{align*}
    K =\begin{bmatrix}
      K_{11} &   &  -\kappa_1\\[2mm]
      & K_{22} &  -\kappa_2\\[2mm]
      -\kappa_1^{\tran} & -\kappa_2^{\tran} &  k_1+k_2+k_3\\
    \end{bmatrix},\, \,
    K_{ii} =k_{i}
    \begin{bmatrix}
      2 & -1  & & &\\
      -1 & 2 &-1 & & \\
      &     \ddots & \ddots  & \ddots&\\
      &    & -1 & 2 & -1 \\
      & & &-1 & 2\\
    \end{bmatrix},
  \end{align*}
  with $\kappa_{i}=\begin{bmatrix}0&
    \ldots&
    0&
    k_{i}
  \end{bmatrix} $ for $i=1,2$.

  Each row has $d$ masses and we consider systems with $2d+1$ masses including

  \begin{align}
    \quad m_{i}&= \left\{ \begin{array}{ll} \frac{1}{10}(2d+1 -  2i) , \quad & i=1,\ldots,500, \\
                            \frac{1}{10}(i-500)+100    , \quad & i=501,\ldots,1000, \\
                            160     , \quad & i=1000,\ldots,2d, \label{mass-Config}
                          \end{array}\right.     \\
    m_{2d+1}&=175. \nonumber
  \end{align}
  The  stiffness values are chosen as
  $  k_1=40, k_2=20$, and $k_3=30$. We assume that the internal damping is modeled as a small multiple of the critical damping~\eqref{Ccrit} with $\alpha= 0.04$.  We would like to remind the reader that, for a given $d$, the coefficient matrix $A(\vvec)$ in~\eqref{Ljap jdba 1} has dimension $n=4d+2$.

  We consider viscosity optimization over three dampers ($k=3$) with viscosities $v_1, v_2$ and $ v_3$ with their positions encoded in
  \begin{equation}\label{damping-geometry}
    B_l=B_r  =\left [\,e_{i_1}, ~~ e_{i_1+\frac{d}{10}}-e_{i_1+\frac{d}{10}+d}, ~~ e_{i_2} \,\right],\quad 1\leq i_1 \leq d,\quad d+1\leq i_2 \leq 2d,\
  \end{equation}
  where $e_{i}$ is the $i$th canonical vector and  the indices $i_1$ and $i_2$
  determine the damping positions. The first damping positions will damp the first row of masses (using a grounded damper). Similarly, the third damps the second row of masses, while the second damper connects both rows of masses of the considered mechanical system.

  We will optimize the viscosity parameters with respect to the average total energy measure introduced in Section~\ref{subsec:application-damped_vibrational_systems}. The optimization problems were solved by using \textsc{Matlab}'s built-in \verb"fminsearch"  with  the starting point $\vvec^0=(100,100, 100)$.  The stopping tolerance for this routine was set to $  10^{-4}$.


  The performance of our new approaches will be illustrated on two damping configurations with different features.

  \begin{itemize}
  \item [a)] In the first case, we consider a small dimensional problem with $d=400$, so we have a system with 801 masses. Here we damp the 9 lowest undamped eigenfrequencies, i.e., $s=9$ in~\eqref{matrix Z}. The damping geometry is determined by~\eqref{damping-geometry} and 20 different damping positions will be considered. These are determined by  the following  indices:
    \begin{align*}
      i_1 \in \{50,   130,  210,   290\}, \quad\text{and}\quad
      i_2\in\{ 460,   540,   620,   700,   780\}.
    \end{align*}

    Due to the small problem dimension ($n=1\,602$), we employ the \REV{SMW+}recycling Krylov technique presented in Section~\ref{Small-scale setting: recycling Krylov methods} to solve this problem. In particular, we use the following parameters for the GCRO-DR method. The threshold on the relative residual norm is $10^{-10}$, the maximum number of iterations allowed is 300, and the number of eigenvectors $s_\ell$ used in the recycling technique is $s_\ell=10$ for all $\ell$. In the preconditioner~\eqref{eq:prec}, the rank $\bar p$ of $\mathcal{N}_1\mathcal{M}_1^{\tran}$ and $\mathcal{N}_2\mathcal{M}_2^{\tran}$ coming from the TSVD of $\mathbf{N}^{\tran}_1\mathbf{L}_0^{-1}\mathbf{M}_1$ and
    $\mathbf{N}^{\tran}_2\mathbf{L}^{-1}_0\mathbf{M}_2$, respectively, is 50.

  \item [b)] In the second case, we increase the problem dimension by considering $d=1000$, thus  here we have a system with 2001 masses and $n=4002$. In this case, we damp the 21  lowest undamped eigenfrequencies, i.e., $s=21$ in~\eqref{matrix Z}. Here we consider 25 damping positions determined by  the following  indices:
    \begin{align*}
      i_1 \in \{50,   250,   450,   650,   850\}, \quad\text{and}\quad
      i_2\in\{ 1150, 1350, 1550, 1750, 1950\}.
    \end{align*}

    We apply the projection framework illustrated in Section~\ref{Large-scale setting: a projection framework} for the solution of this larger problem.
    We would like to mention that a Lyapunov equation whose coefficient matrix is of dimension $4002$ as in this case is usually not considered to be large scale in the matrix equation community. On the other hand, the huge number of Lyapunov equations we need to solve within the optimization procedure makes our novel projection framework very appealing. In Algorithm~\ref{EKSM_opt} we employ $\epsilon= 10^{-8}$ and $m_{\max}=120$. The projected equations~\eqref{eq:param_projected} are solved by means of our \REV{SMW+}recycling Krylov approach, adopting the same parameters as in case a) above.

  \end{itemize}

  We first focus on case a) described above and in
  Figure~\ref{RelErrorViscCasea} we report the relative errors obtained by our \REV{SMW+}recycling Krylov method. The relative errors in the optimal viscosities were calculated by $\|\vvec^*-\vvec\|/\|\vvec\|$, where $\vvec$ and $\vvec^*$ denote the optimal viscosity vectors calculated by the \REV{SMW+}recycling Krylov method and the \textsc{Matlab}'s function \lyap{}, respectively. Similarly, the relative errors in the  average total energy are calculated by $|\trace ({X(\vvec)}^{*})-\trace ( X(\vvec))|/\trace ({X(\vvec)}^{*})$, where $\trace ({X(\vvec)}^{*})$ is once again the optimal trace for the given configuration obtained by the \textsc{Matlab}'s function \lyap{}, and $\trace (X(\vvec))$ is the optimal trace calculated by our \REV{SMW+}recycling Krylov scheme.

  \begin{figure}[t!]
    \begin{center}
      \includegraphics[width=\textwidth]{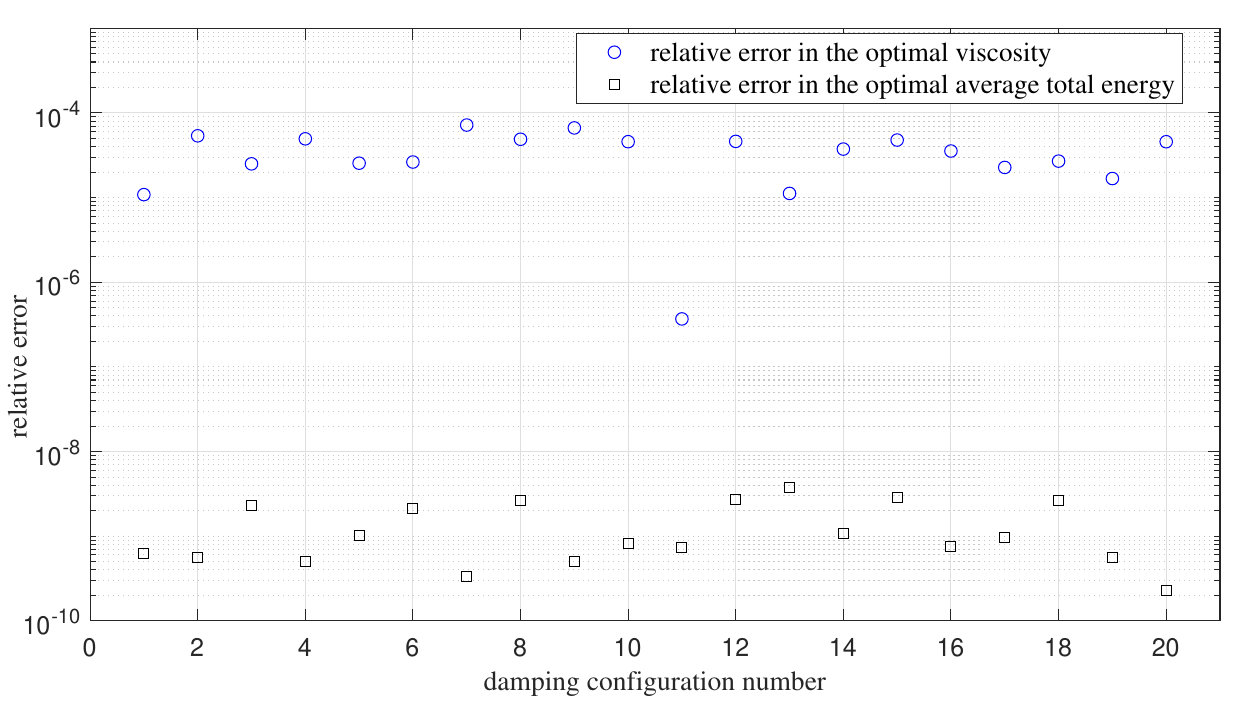}
    \end{center}%
    \caption{Example~\ref{ex1}, the case a). Relative errors in the total average energy (squares) and in the viscosity (circles) at optimal gains for the \REV{SMW+}recycling Krylov method.}%
    \label{RelErrorViscCasea}
  \end{figure}

  From the results reported in Figure~\ref{RelErrorViscCasea} we can notice that our novel \REV{SMW+}recycling Krylov approach leads to a solution process able to achieve very small errors in the average total energy. Also, the computed optimal viscosities turn out to be rather accurate, attaining an error of the same order of the threshold used for {\tt fminsearch}.
  Moreover, our novel \REV{SMW+}recycling Krylov strategy is very efficient. In particular, in this example, it accelerated the overall optimization process approximately 2.7 times (in the average case). Figure~\ref{RelErrorFunCasea} shows a precise acceleration ratio for each considered configuration. The displayed  time ratio is the ratio between the total time needed for the direct calculation of the average total energy by using \textsc{Matlab}'s function \lyap{} and the total time needed for the viscosity optimization by using the \REV{SMW+}recycling Krylov method.

  We would like to mention that, for this example, we tried to use also GCRO-DR with no preconditioning within our recycling Krylov technique. In the best case scenarios, plain GCRO-DR managed to converge by performing a larger number of iterations: $\mathcal{O}(10^2)$ iterations to be compared to the $\mathcal O(10^1)$ iterations performed in the preconditioned case. Such large number of iterations remarkably worsened the overall performance in terms of computational time. In some other cases, unpreconditioned GCRO-DR did not achieve the prescribed level of accuracy, thus jeopardizing the convergence of the outer minimization procedure. This shows that the preconditioning operator $\mathcal{P}$ described in~\eqref{eq:prec} works well, even though more performing preconditioners can certainly be designed.

  \begin{figure}[t!]
    \begin{center}
      \includegraphics[width=\textwidth]{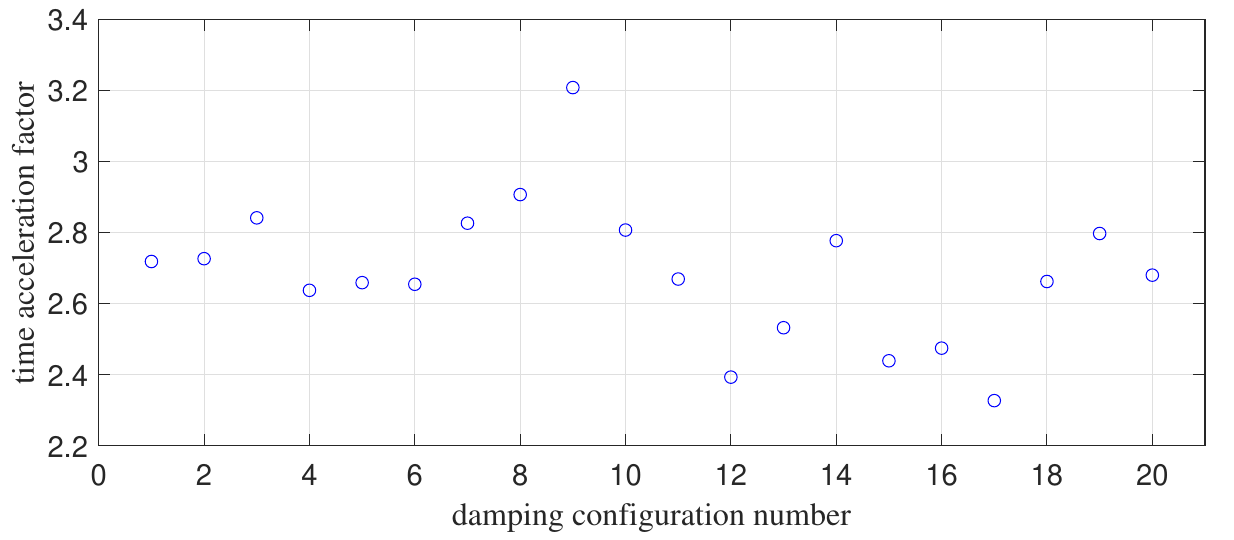}
    \end{center}%
    \caption{Example~\ref{ex1}, case a). Acceleration factors in the overall minimization procedure attained by employing our novel recycling Krylov approach.}%
    \label{RelErrorFunCasea}
  \end{figure}

  We now turn our attention to case b).
  As already mentioned, having coefficient matrices of dimension 4002 is often considered as working in the small-scale setting in the matrix equation literature so that dense linear algebra solvers may be preferred in this case. In this example, we would like to show that employing a projection framework is largely beneficial also in this scenario. To this end, we compare our fresh projection framework with the \REV{SMW+}recycling Krylov technique we propose in Section~\ref{Small-scale setting: recycling Krylov methods}. Indeed, we showed above that the latter method is able to achieve small errors while accelerating the overall solution process.
  Figure~\ref{RelErrorViscCaseb} shows the relative errors achieved by our
  projection framework with respect to the \REV{SMW+}recycling Krylov method.
  We can notice that the error in the optimal viscosities is still of the same order of magnitude of the threshold used within {\tt fminsearch} whereas the relative errors in the optimal average total energy norm are a little higher than the ones reported in
  Figure~\ref{RelErrorViscCasea}, even though still satisfactory. Moreover, the
  projection framework method accelerated the optimization process by approximately 23.3 times and Figure~\ref{RelErrorFunCaseb} showcases a precise acceleration time ratio for all the 25 configurations, compared to the \REV{SMW+}recycling Krylov method. The main reason why our novel projection method performs so well in terms of computational time lies in the fact that we basically use the same approximation subspace for every $\vvec_\ell$. In particular, once we construct a sufficiently large $\mathbf{EK}_m^\square(A_0,P)$ for the first equation, namely $\Delta_{m,\vvec_0}$ meets our accuracy demand, we keep using the same approximation subspace for all the subsequent equations, by expanding it very few times instead of computing a new subspace from scratch. For most of the damping configurations we considered, $\mathbf{EK}_m^\square(A_0,P)$ does not need to be expanded at all after its construction during the solution of the first equation.  Only for three configurations -- $(i_1,i_2)\in\{(250,1\,750),(450,1\,550),(450,1\,950)\}$ -- $\mathbf{EK}_m^\square(A_0,P)$ needed to be expanded a couple of times during the optimization procedure. In particular, for $(i_1,i_2)=(250,1\,750)$, the dimension of $\mathbf{EK}_m^\square(A_0,P)$ after the solution of the first equation was 708. This space has been expanded twice during the online optimization step, getting a space of dimension 720 and then 732. Similarly, for $(i_1,i_2)=(450,1\,950)$ we started the online phase with a space of dimension 780, which got expanded twice getting a space of dimension 792 first, and 804 later. For $(i_1,i_2)=(450,1\,550)$, $\mathbf{EK}_m^\square(A_0,P)$ was expanded only once, passing from a space of dimension 768 to a space of dimension 780.
  In general, for this example, the dimension of $\mathbf{EK}_m^\square(A_0,P)$ averaged over all the configurations is about $749$ with a minimum and maximum dimensions equal to $708$ and $852$, respectively.


  \begin{figure}[t!]
    \begin{center}
      \includegraphics[width=\textwidth]{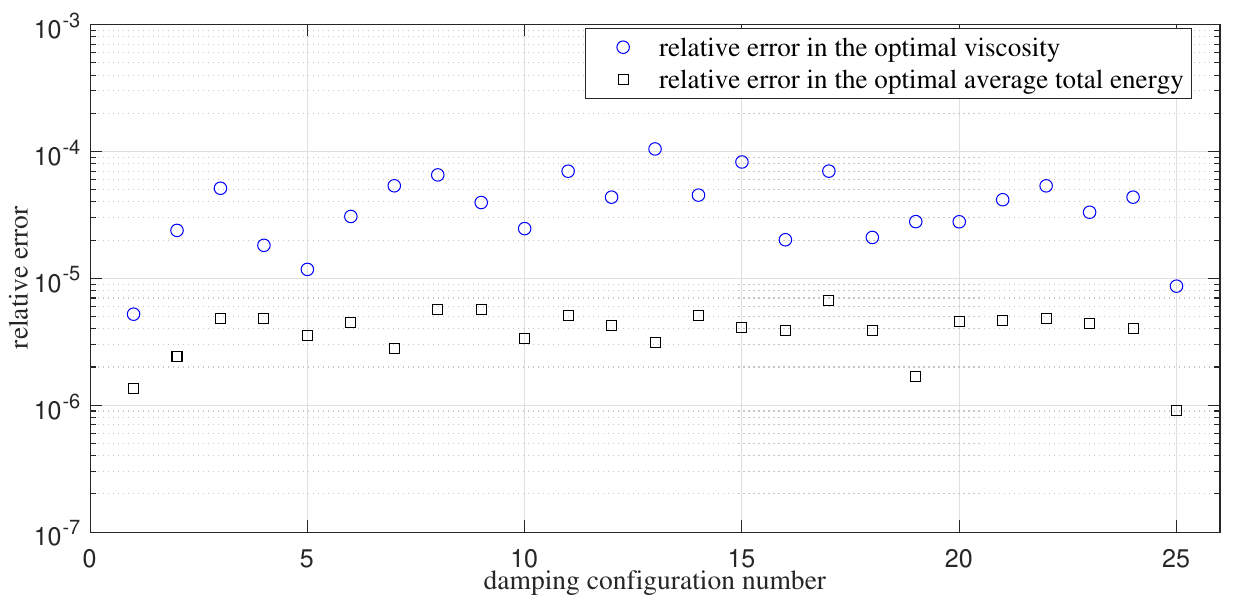}
    \end{center}%
    \caption{Example~\ref{ex1}, case b). Relative errors in the total average energy (squares) and in the viscosity (circles) at optimal gains for the projection framework.}%
    \label{RelErrorViscCaseb}
  \end{figure}

  \begin{figure}[t!]
    \begin{center}
      \includegraphics[width=\textwidth]{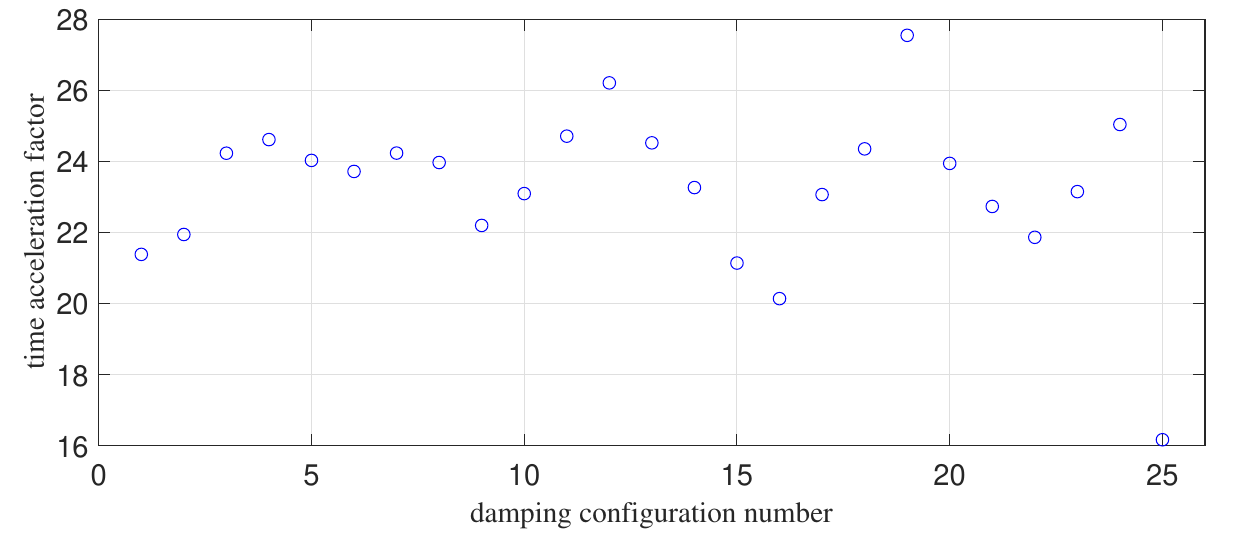}
    \end{center}%
    \caption{Example~\ref{ex1}, case b). Acceleration factors in the overall minimization procedure attained by employing our new projection framework.}%
    \label{RelErrorFunCaseb}
  \end{figure}

\end{example}

In the next numerical examples, we are going to illustrate the following two cases of analysis of multi-agent systems that are of interest:
\begin{itemize}
\item For each agent $i$ we choose $H=e_i$. Then we study the influence of the variation of the dynamics of the agent $i$ by calculating the $H_2$-norm of the corresponding systems. By applying this procedure to all agents,  we can study the structural importance of each agent in the system.
\item We choose $H=I_m$. For a given set of subsets of agents $\left\{J_1,\ldots,J_q  \right\}$, $J_i\subset \left\{ 1,\ldots,m \right\}$, and a given set of parameters of agents in $J_i$, we study the influence of these parameters on the dynamics of the overall system. By calculating the corresponding $H_2$-norms, we can study the influence of these parameters on the dynamics of the whole system when all agents are exposed to independent external disturbances.
\end{itemize}
See Examples~\ref{ex2} and~\ref{ex3} for further details.

\begin{example}\label{ex2}
  For the sake of simplicity, here
  we investigate agent dynamics that are defined by $2\times 2$ state matrices.
  We consider $m=196$ agents.
  The state matrices of all agents are given by
  \[
    A^{(i)}=\begin{bmatrix}
      -10&5   \\5&-8
    \end{bmatrix},
    \quad \text{for}\quad i=1,\ldots, m.
  \]
  By following the notation of Section~\ref{subsec:multi-agent systems}, the other components of the system are
  \[ K^{(i)}=\begin{bmatrix}
      0.3 &0 \\ 0&0.2
    \end{bmatrix}, \quad B^{(i)}=(C^{(i)})^{\tran}=\begin{bmatrix}
      1 &1 \\ -1&1
    \end{bmatrix},\quad \mbox{for} \quad i=1,\ldots, m.
  \]
  Therefore, the $2\times 2$ blocks of the system matrix $\mathbf{A}\in \mathbb{R}^{2m\times 2m}$ are given by
  \begin{align*}
    \mathbf{A}_{ij} & =\delta_{i,j}A^{(i)} -\ell_{ij}B^{(i)}K^{(i)}C^{(j)}, \quad i,j\in \left\{ 1,\ldots,m \right\}.
  \end{align*}
  An illustration of the underlying topology can be seen in Figure~\ref{Ex1Topology}.

  We would like to analyze the influence of parameter variations in different agents. This means that for a fixed $k$-th agent to be altered, we consider a low-rank update of $\mathbf{A}$ aimed at modifying only its $k$-th diagonal $2\times 2$ block. In particular, given the parameters $\vvec=(v_{1}, v_{2}, v_{3})$,  the matrix $\mathbf{A}^{(k)}(\vvec)$ corresponds to a low-rank update of $\mathbf{A}$ of the following form

  \begin{equation}\label{AkOneAgentCase}
    \mathbf{A}^{(k)}(\vvec)=\mathbf{A}-B_l
    \mathop{\mathrm{diag}}( v_1 ,v_2, v_2 , v_3 )B_r^{\tran},
  \end{equation}
  where $B_l,B_r \in \RR^{2m\times 4 }$ are determined by the agent index $k$. It holds,
  \[
    B_l(2(k-1)+1:2k,1:4) =\begin{bmatrix}
      1 & 1 & 0 & 0 \\
      0 & 0 & 1 & 1 \\
    \end{bmatrix},\quad
    B_r(2(k-1)+1:2k,1:4)
    =\begin{bmatrix}
      1 & 0 & 1 & 0 \\
      0 & 1 & 0 & 1\\
    \end{bmatrix},
  \]
  whereas all other entries of $B_l$ and $B_r$ are zero. In this example, we considered  the influence of the variation of the dynamics of one agent by calculating the $H_2$-norm of a corresponding system when only this agent is externally disturbed. Thus, in the case of the matrix $\mathbf{A}^{(k)}$ (that corresponds to analyzing the $k$-th agent), the corresponding matrix $\mathbf{E}$ is given by $\mathbf{E}= e_k\otimes I_n$, for $k=1,\ldots,m$.

  For all different agents, we will analyze the dependence of the system to the parameter variation $v_1$, $v_2 $, and $v_3$ of the $k$-th agent, $k=1,\ldots,m$. In particular, the low-rank update of the system matrix given by~\eqref{AkOneAgentCase} influences the $k$-th block in the following way. The diagonal elements are altered by the parameters $v_1$ and $v_3$, while the off-diagonal elements are symmetrically modified by the parameter $v_2$. All the parameters $v_1$, $v_2 $, and $v_3$ will be varied between -9.9 and 10.1 with step $1$.

  In this example we will modify all the agents, namely we consider all agent's indices $k=1,\ldots,m$.

  We note that some instances of $\mathbf{A}^{(k)}(\vvec)$ turned out to be non-stable. \REV{In particular, we have computed the rightmost eigenvalue $\lambda$ of $\mathbf{A}^{(k)}(\vvec)$ and discarded those configurations for which $\text{Re}(\lambda)\geq 0$.} By doing so, for this particular example, the total number of considered parameters is equal to 1\,400\,328\footnote{\REV{This number is obtained by multiplying all the possible values of $v_1$, $v_2$, and $v_3$, namely $21^3$, by all the adopted selections of $k$, i.e., $m$, and then subtracting the number of unstable cases ($414\,828$).}}. This means that we needed to calculate the $H_2$-norm, and thus solve a Lyapunov equation, a significant number of times. In particular, for fixed $v_1$, $v_2$, $v_3$, we are interested in identifying for which $k$, $k=1,\ldots, m $, the $H_2$-norm of the underlying system over described parameter variations is maximal. This provides an insight on the importance of the considered agent.

  Table~\ref{Table196agents} presents the average computational time needed by different methods to perform the task described above. 
  We report the average required time by our novel solvers, namely the \REV{SMW+}recycling Krylov method and the projection framework (denoted in the table by \REV{SMW+}rec. Krylov and proj.\ framework, respectively). The computational parameters of our \REV{SMW+}recycling Krylov method are as in Example~\ref{ex1}. The only exceptions are in the GCRO-DR threshold and the rank $\bar p$ of the TSVD $\mathcal{N}_1\mathcal{M}_1^{\tran}$ and $\mathcal{N}_2\mathcal{M}_2^{\tran}$ used in the preconditioner. Here, we use $10^{-8}$ and $\bar p=5$, respectively. For Algorithm~\ref{EKSM_opt} we used $\epsilon=10^{-10}$ and $m_{\max}=200$. The projected equations are solved by our \REV{SMW+}recycling Krylov method with the same setting we have just described above.


  In the same table, we also document the relative errors achieved by the different algorithms. To this end, we considered the results obtained by using \lyap{} as \emph{exact}. As we can see, all the approaches achieve a very small relative error. On the other hand, the projection framework outperforms all other approaches in terms of computational time; being one order of magnitude faster than all the other algorithms we tested. Also in the multi-agent system framework, reusing the same subspace $\mathbf{EK}_m^\square(A_0,P)$ is key for our projection method to be successful. For this example, the extended Krylov subspace generated by our routine is very small for all the configurations we tested. In particular, the space dimensions range from $8$ to $56$ with an average value equals to $29.08$.

  \begin{table}[t!]
    \begin{center}
      \begin{center}
        \begin{tabular}{|c||c|c|}
          \hline
          & average required time (s) & average relative error \\ \hline \hline \xrowht[()]{7pt}
          \lyap{}          &   625.02   &  -  \\\hline \xrowht[()]{7pt}
          \REV{SMW+}rec. Krylov      &   678.05   & $3.575\cdot 10^{-12}$  \\\hline \xrowht[()]{7pt}
          proj.\ framework &  17.48    & $3.675\cdot 10^{-12}$    \\
          \hline
        \end{tabular}
      \end{center}
      \caption{Example~\ref{ex2}. Computational time and relative errors.}%
      \label{Table196agents}
    \end{center}
  \end{table}

  \begin{figure}[t!]
    \begin{center}
      \includegraphics[width=.8\textwidth]{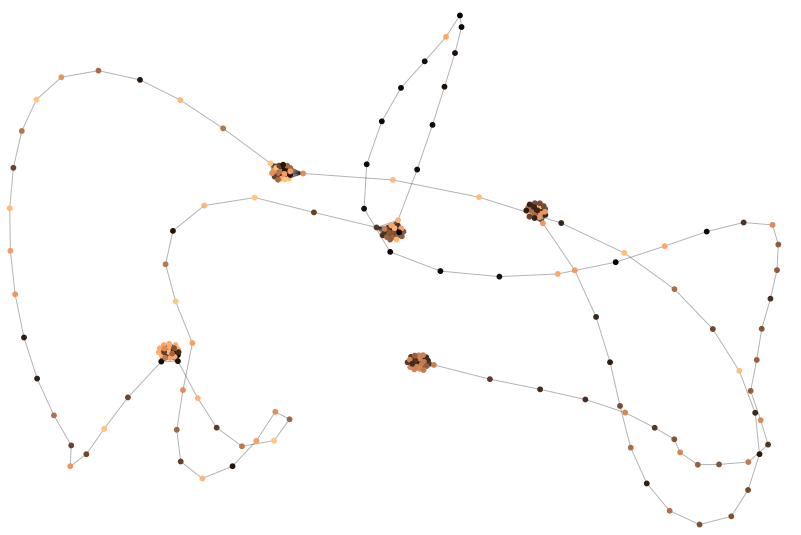}
    \end{center}%
    \caption{Example~\ref{ex2}. Maximal $H_2$-norm for different agents with respect to the parameter variations we performed. Darker colors correspond to the agents where the maximal $H_2$-norm is larger.}%
    \label{Ex1Topology}
  \end{figure}

  Figure~\ref{Ex1Topology} illustrates the topology of the problem and provides information regarding the maximal $H_2$-norm we computed by varying the parameter of the system. The node colors illustrate this latter aspect. In particular, darker colors correspond to the agents where the maximal $H_2$-norm turns out to be larger. It can be seen that certain groups of agents result in much larger $H_2$-norms. Therefore, they are much more important for the considered multi-agent system, as altering those agents may lead to an (almost) unstable system. Thanks to our novel solvers, this kind of analysis can now be accurately carried out at acceptable cost.
\end{example}

\begin{example}\label{ex3}
  In this example, we still consider agent dynamics defined by $2\times 2$ state matrices of dimension $m$. We aim at analyzing the parameter variation in off-diagonal elements of consecutive agents. The system matrices $A^{(i)}$, $B^{(i)}$, $C^{(i)}$, $K^{(i)}$,  and $\mathbf{A}$ are defined as in Example~\ref{ex2} but here we consider $m =200$.

  The Laplacian matrix $L$ of the underlying graph is given by Algorithm~\ref{TopologyL} and the corresponding matrices $B_l$ and $B_r$ determining the low-rank perturbation are defined in the following way. For a given odd index $1\leq k\leq 2m-3$, $B_l,B_r \in \RR^{2m\times 4 }$ are
  \[
    B_l(k:k+3,1:4) =
    \begin{bmatrix}
      0 & 1 & 0 & 0 \\
      1 & 0 & 0 & 0 \\
      0 & 0 & 0 & 1 \\
      0 & 0 & 1 & 0
    \end{bmatrix},\quad
    B_r(k:k+3,1:4)=I_4,
  \]

  and all other entries of $B_l$ and $B_r$ are zero.

  \begin{algorithm}[t]
    \DontPrintSemicolon
    \SetKwInOut{Input}{input}\SetKwInOut{Output}{output}
    \Input{number of agents $m$}
    \Output{the Laplacian matrix $L$ }%
    $r= [m / 20 + m / 50:m / 10,\, m / 4 + m / 20:m / 4 + m / 10,\, m / 2 + m / 20 + 2:m / 2 + m / 20 + m / 50,\ldots$
    $\phantom{r= [} m / 2 + m / 10 + 2:m / 2 + m / 10 + m / 20,\, m - m / 10:m - m / 20]\in \mathbb{R}^{i_r}$\;
    $h=[1:m / 20,\, m / 10 + m / 40:m / 10 + m / 20,\, m / 2:m / 2 + m / 20,\, m - 1]\in \mathbb{R}^{i_h}$\;
    \For {$i=1: m-2$}{
      $L(i,i+1)=-1  $;\, $ L(i,m)=-1 $;
    }

    $L(1,m-1)=-1$;\, $L(m-1,m)=-1$;\;

    \For {$i=1:i_r$}{
      \If {$r(i)==(m-1)$}{
        $L(1,m-1)=0$;
      }
      \Else{
        $L(r(i),r(i)+1)=0$;
      }
    }
    $L(h,m)=0$; $L=L+L^{\tran}$;\;
    \For {$i=1:m$}{
      $L(i,i) = -\sum_{j=1}^{m}(L(i,j))$;
    }
    \caption{Construction of the Laplacian matrix $L$ in Example~\ref{ex3}.}%
    \label{TopologyL}
  \end{algorithm}

  The modified system matrix will depend on two parameters $v_1$ and $v_2$ and it is defined as
  \begin{equation*}
    \mathbf{A}^{(k)}(\vvec)=\mathbf{A}-B_l
    \mathop{\mathrm{diag}}( v_1 ,v_1, v_2 , v_2 )B_r^{\tran}.
  \end{equation*}
  This means that for a fixed index $k$, the off-diagonal elements of the diagonal blocks determined by $k$ and $k+1$ are modified by $v_1$ and $v_2$, respectively. The parameters $v_1$ and $v_2$ vary between -4.9 and 14.6 with step $0.5$. As before, we only consider the stable instances of the matrices $\mathbf{A}^{(k)}$ for our purposes.

  We will study the behavior of the $H_2$-norm of the underlying system by varying the parameter $k\in \{61 ,  141 ,  221  , 301\} $. The total number of parameter configurations $\vvec=(v_1,v_2)$ for which $\mathbf{A}^k(\vvec)$ is stable is equal to $6\,044$\footnote{\REV{As before, this number is obtained by multiplying all the possible values of $v_1$ and $v_2$, namely $40^2$, by all the adopted selections of $k$, i.e., $4$, and then subtracting the number of unstable cases (356).}}.
  Moreover, in this example, we considered the matrix $\mathbf{E}=I_{2m}$ for all cases, meaning that we measured the $H_2$-norm of the corresponding systems when all agents were independently externally disturbed.

  We test the same routines as in Example~\ref{ex2}, namely \lyap{}, 
  and our novel schemes, with the same setting as before.

  We start by reporting the relative errors and the computational timings attained in the computation of the $H_2$-norms for all the considered parameters; see Table~\ref{AktwoAgentCase}. The relative errors are calculated with respect to $H_2$-norm obtained by using \lyap{}.

  \begin{table}[t!] \begin{center}
      \begin{center}\footnotesize
        \begin{tabular}{|r||cccc|cccc|}
          \hline
          \multirow{2}{*}{} &
                              \multicolumn{4}{c|}{total required time(s)} & 
                                                                            \multicolumn{4}{c|}{average relative error}   \\
          $k$  & 41 &  121 &  201  & 281 & 41 &  121 &  201  & 281 \\
          \hline \hline \xrowht[()]{7pt}
          \lyap{}  & 165.51  & 141.36  & 157.43 & 164.64 & -&-&-&-\\ \hline \xrowht[()]{7pt}
          \REV{SMW+}rec. Krylov &    257.49 & 242.93 & 255.52 & 265.59 &  7 $\cdot 10^{-14}$&    1.03  $\cdot 10^{-13}$& 2.07 $\cdot 10^{-13}$&      1.7  $\cdot 10^{-13}$\\  \hline \xrowht[()]{7pt}
          proj.\ framework & 15.23  &  9.55  & 10.71 &  25.31 & 9.79 $\cdot 10^{-14}$&  7.64 $\cdot 10^{-14}$&    2.18   $\cdot 10^{-13}$&  1.28 $\cdot 10^{-11}$\\
          \hline
        \end{tabular}
      \end{center}
      \caption{Example~\ref{ex3}. Computational time and relative errors.}%
      \label{AktwoAgentCase}
    \end{center}
  \end{table}

  From the results in Table~\ref{AktwoAgentCase} we can see that our \REV{SMW+}recycling Krylov scheme always attains very small relative errors. On the other hand, it turns out to be not very competitive in terms of computational timings. A finer tuning of the GCRO-DR parameters and the preconditioning operator may lead to some improvements in the performance of our scheme. Our projection framework outperforms all other approaches in terms of computational timing.

  The extended Krylov subspaces constructed by our projection method turn out to be a little larger in this example than the ones in Example~\ref{ex2}. In particular, the smallest and largest subspaces have dimension $12$ and $96$, respectively, whereas the average space dimension is $49.71$.

  \begin{figure}[t!]
    \begin{center}
      \includegraphics[width=\textwidth]{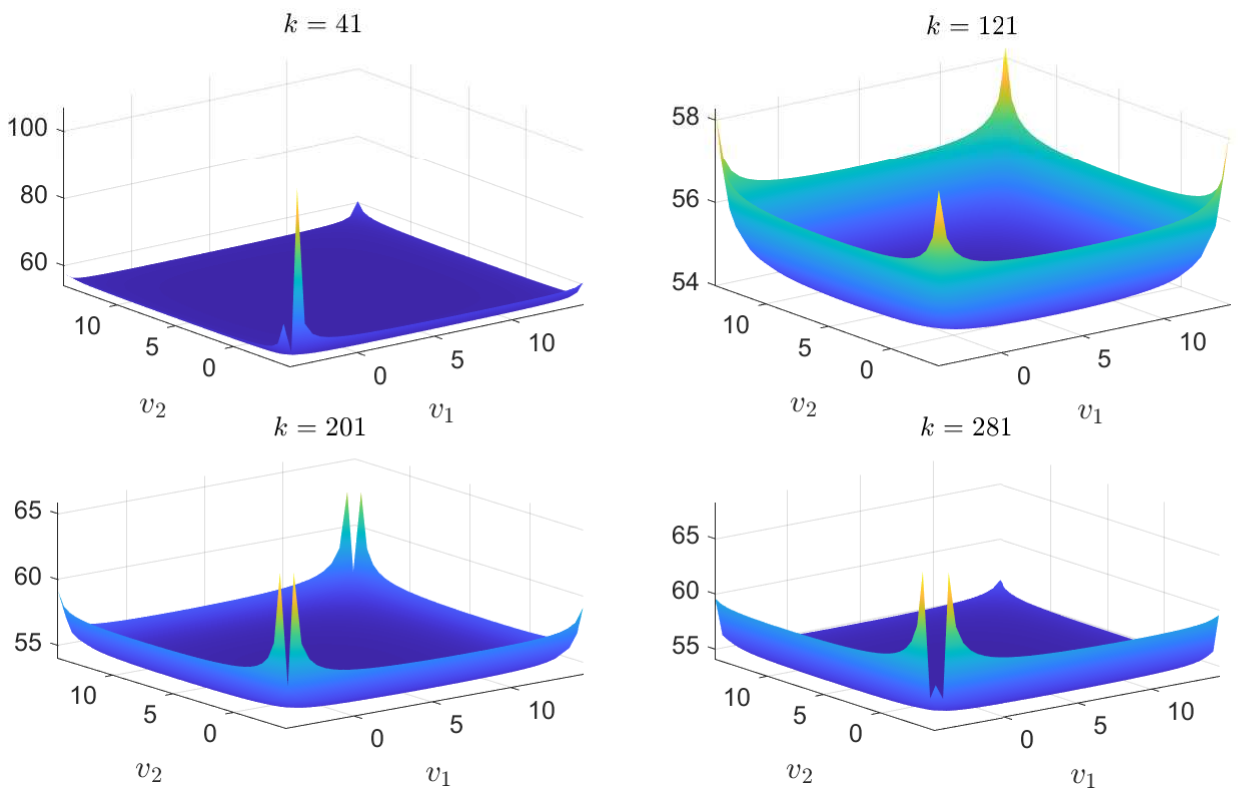}
    \end{center}%
    \caption{Example~\ref{ex3}: magnitude of the $H_2$-norm.}%
    \label{Ex2Surface}
  \end{figure}

  In Figure~\ref{Ex2Surface} we display the magnitude of the $H_2$-norm by  varying $\vvec$ for a fixed $k$. It can be seen how certain parameter values have a much greater impact on the $H_2$-norm of the system, thus emphasizing the important role of the corresponding pairs of agents. Once again, thanks to our novel solution processes, such analysis can be carried out in very few seconds on a standard laptop.

\end{example}

\section{Conclusions}\label{Conclusions}

We proposed two different, efficient, and accurate methods for solving sequences of parametrized Lyapunov equations. The \REV{SMW+}recycling Krylov approach is well-suited for small dimensional problems and is able to provide solutions achieving a very small relative errors. The proposed projection framework relies on the extended Krylov subspace method, and it makes use of the aforementioned \REV{SMW+}recycling Krylov technique to solve the projected equations.
Even though this second algorithm is tailored to large-scale problems, our numerical results show that it is extremely competitive also for medium-sized problems, especially if the number of Lyapunov equations to be solved is very large as it happens in the application settings we studied. We showed that the projection framework
is able to speed up the entire solution process by an order of magnitude.
Expensive analyses like viscosity optimization for vibrational systems and output synchronization of multi-agent systems are, thus, now affordable.

\section*{Acknowledgements}
The first author is a member of the INdAM Research Group GNCS\@.
His work was partially supported by
the funded project GNCS2023 ``Metodi avanzati per la risoluzione di PDEs su griglie strutturate, e non''  (CUP\_E53C22001930001).

The work of the second author was  supported in parts by Croatian Science Foundation under the project  `Vibration Reduction in Mechanical Systems' (IP-2019-04-6774).

\section*{Conflict of interest}
 The authors declare no competing interests.

\bibliographystyle{siamplain}
\bibliography{bibliogr}
\end{document}

%% file: generic.tex
\usepackage{amsmath,amssymb}
\usepackage{mathtools}
\usepackage{tikz}
\usepackage{pgfplots}
\usepackage{pgfplotstable}
\usepackage{geometry}
\usepackage{color}
\usepackage{booktabs}
\usepackage{framed}
\pgfplotsset{compat=1.9}

\pgfplotstableset{
	every head row/.style={before row=\toprule,after row=\midrule},
	clear infinite
}

\usepackage{amsopn}

\DeclareMathOperator{\diag}{diag}

\renewcommand{\leq}{\leqslant}
\renewcommand{\geq}{\geqslant}